\documentclass[11pt]{amsart}

\usepackage{epsf,amssymb,amsmath,overpic,amscd,psfrag,stmaryrd,times,mathrsfs, euscript,thmtools, amsbsy, multicol, mathtools}

\usepackage{caption}
\usepackage{subcaption}
\usepackage[all]{xy}
\usepackage{hyperref, stmaryrd}
\usepackage[dvipsnames]{xcolor}
\usepackage{enumitem, bbm}
\usepackage{tikz,tikz-cd}
\usepackage{array}

\tikzcdset{scale cd/.style={every label/.append style={scale=#1},
    cells={nodes={scale=#1}}}}

\usepackage{todonotes}
\usetikzlibrary{decorations.markings}

\newtheorem{theorem}{Theorem}[section]

\newtheorem{claim}[theorem]{Claim}
\newtheorem{conjecture}[theorem]{Conjecture}

\newtheorem{lemma}[theorem]{Lemma}
\newtheorem{proposition}[theorem]{Proposition}

\theoremstyle{definition}
\newtheorem{definition}[theorem]{Definition}
\newtheorem{example}[theorem]{Example}

\newtheorem{remark}[theorem]{Remark}

\newtheorem{question}[theorem]{Question}

\numberwithin{equation}{subsection}

\DeclareMathAlphabet{\mathpgoth}{OT1}{pgoth}{m}{n}

\DeclareMathAlphabet{\mathpzc}{OT1}{pzc}{m}{it}

\newcommand{\C}{\mathbb{C}}
\newcommand{\R}{\mathbb{R}}
\newcommand{\Z}{\mathbb{Z}}

\newcommand{\N}{\mathbb{N}}

\newcommand{\be}{\begin{enumerate}}
\newcommand{\ee}{\end{enumerate}}
\newcommand{\op}{\operatorname}

\newcommand{\arr}{\overrightarrow}

\newcommand{\cbu}{\color{black}}

\newcommand{\cb}{\color{black}}

\DeclareMathOperator{\ind}{Ind}

\DeclareMathOperator{\crit}{Crit}
\newcommand{\hess}[1]{\text{Hess}_{#1}}

\newcommand{\morsedbar}{\mathcal L}

\newcommand{\iT}[1]{#1}

\newcommand{\obstructionSectionS}{\mathfrak{s}}

\newcommand{\obstructionBundle}{\mathcal O}

\newcommand{\chartData}[1]{\mathcal{C}_{\iT{I}}}

\DeclareMathOperator{\stab}{stab}

\input pdf-trans
\newbox\qbox
\def\usecolor#1{\csname\string\color@#1\endcsname\space}
\newcommand\bordercolor[1]{\colsplit{1}{#1}}
\newcommand\fillcolor[1]{\colsplit{0}{#1}}
\newcommand\outline[1]{\leavevmode%
  \def\maltext{\mydelim #1\mydelim}%
  \setbox\qbox=\hbox{\maltext}%
  \boxgs{Q q 2 Tr \bbthickness\space w \fillcol\space \bordercol\space}{}%
  \copy\qbox%
}
\newcommand\mathbh[1]{\def\mydelim{$}\outline{#1}}

\newcommand\colsplit[2]{\colorlet{tmpcolor}{#2}\edef\tmp{\usecolor{tmpcolor}}%
  \def\tmpB{}\expandafter\colsplithelp\tmp\relax%
  \ifnum0=#1\relax\edef\fillcol{\tmpB}\else\edef\bordercol{\tmpC}\fi}
\def\colsplithelp#1#2 #3\relax{%
  \edef\tmpB{\tmpB#1#2 }%
  \ifnum `#1>`9\relax\def\tmpC{#3}\else\colsplithelp#3\relax\fi
}
\bordercolor{black}
\fillcolor{white}
\newcommand\bbthickness{.5}

\newcommand{\cP}[1]{\mathsf{#1}} 

\newcommand{\A}{\mathscr A}
\newcommand{\D}{\mathscr D}

\newcommand{\p}{\partial}
\newcommand{\fatD}{\mathbh{\partial}}

\newcommand{\M}[3]{\mathcal{M}(#1;\cP{#2},\cP{#3})}
\newcommand{\m}[2]{\mathcal{M}(\cP{#1},\cP{#2})}
\newcommand{\lm}{\mathcal{M}}

\newcommand{\orientation}[1]{\boldsymbol{o}({#1})}

\newcommand{\vt}{t}

\makeatletter
\newcommand{\labitem}[2]{%
\def\@itemlabel{#1}
\item
\def\@currentlabel{#1}\label{#2}}
\makeatother

\begin{document}

\title[Equivariant Morse Homology via Broken Trajectories]
{Equivariant Morse Homology for Reflection Actions via Broken Trajectories}

\author{Erkao Bao}
\address{School of Mathematics, University of Minnesota, Minneapolis, MN 55455}
\email{bao@umn.edu} \urladdr{https://erkaobao.github.io/math/}

\author{Tyler Lawson}
\address{School of Mathematics, University of Minnesota, Minneapolis, MN 55455}
\email{tlawson@math.umn.edu} 
\urladdr{https://www-users.cse.umn.edu/~tlawson/}

\author{Lina Liu}
\address{School of Mathematics, University of Minnesota, Minneapolis, MN 55455}
\email{liu02226@umn.edu}
\urladdr{https://sites.google.com/view/linaliu/}


\thanks{Erkao Bao is supported by NSF Grants DMS-2404529.}
\thanks{Tyler Lawson is supported by NSF Grants DMS-2208062.}
\thanks{Lina Liu is supported by NSF GRFP Grants 2237827.}

\begin{abstract}
    We consider a finite group $G$ acting on a manifold $M$. 
    {\cbu According to \cite{wasserman1969equivariant, bao2024morse}, a generic equivariant  function on $M$ is Morse.
    For any equivariant Morse function,} there does not always exist an equivariant metric $g$ on $M$ such that the pair $(f,g)$ is Morse-Smale. Here, the pair $(f,g)$ is called Morse-Smale if the descending and ascending manifolds intersect transversely. The best possible metrics $g$ are those that make the pair $(f,g)$ \emph{stably Morse-Smale}. 
    
    A diffeomorphism $\phi: M \to M$ is a \emph{reflection} if $\phi^2 = \op{id}$ and the fixed point set of $\phi$ forms a codimension-one submanifold (with $M \setminus M^{\op{fix}}$ not necessarily disconnected).
    In this note, we focus on the special case where the group $G = \{\op{id}, \phi\}$. We show that the condition of being \emph{stably Morse-Smale} is generic for metrics $g$. Given a stably Morse-Smale pair, we introduce a canonical equivariant Thom-Smale-Witten complex by counting certain broken trajectories.
    
    This has applications to the case when we have a manifold with boundary and when the Morse function has critical points on the boundary. 
    We provide an alternative definition of the Thom-Smale-Witten complexes, which are quasi-isomorphic to those defined by \cite{kronheimer2007Monopoles}. 
    
    We also explore the case when $G$ is generated by multiple reflections. As an example, we compute the Thom-Smale-Witten complex of an upright higher-genus surface by counting broken trajectories.
\end{abstract}

\maketitle

\setcounter{tocdepth}{2}


\section{Introduction and main results}

Let $M$ be a closed manifold with a finite group $G$ acting on it via diffeomorphisms. Consider an equivariant Morse function $f$, which exists generically by \cite{wasserman1969equivariant,bao2024morse}. We begin by reviewing the equivariant Morse homologies defined in \cite{bao2024morse}.

For any critical point $\cP{p} \in \crit(f)$, let $H = \stab(\cP{p})$ be the stabilizer subgroup of $\cP{p}$. 
Let $(T_\cP{p} M)^H$ be the $H$-equivariant subspace of $T_\cP{p} M$.
We define $T'_\cP{p} M$ to be the kernel of the averaging map $T_\cP{p} M \to (T_\cP{p} M)^H$, $$v \mapsto \frac{1}{|H|}\sum_{\sigma \in H} d \sigma \cdot v.$$
Then we have the decomposition $T_\cP{p} M = (T_\cP{p} M)^H \oplus T'_\cP{p} M$.
\begin{definition}[Stable critical point]
    A critical point $\cP{p} \in \crit(f)$ is called \emph{stable} if the Hessian $\hess{f}(\cP{p})$ is positive definite on $T'_\cP{p} M$.
\end{definition}

\begin{definition}[Stable Morse function]
    A Morse function $f$ is called \emph{stable} if all its critical points are stable.
\end{definition}

As shown in \cbu Theorem 1.7 of \cb \cite{bao2024morse}, any equivariant Morse function $f$ can be perturbed near its unstable critical points to obtain a stable Morse function $f^\epsilon$ via a $C^0$-small perturbation. Let $g$ denote an equivariant metric on $M$.

\begin{definition}[Morse-Smale]
A pair $(f,g)$ is Morse-Smale if the descending manifold $\mathscr D_{\cP{p}}$ intersects the ascending manifold $\mathscr A_\cP{q}$ transversely in $M$ for all critical points $\cP{p}, \cP{q} \in \crit(f)$.   
\end{definition}

\begin{theorem}[\cbu Theorem 1.8 \cb \cite{bao2024morse}]
    For a stable equivariant Morse function $f$, a generic equivariant metric $g$ makes the pair $(f,g)$ Morse-Smale.
\end{theorem}

\cbu 
The equivariant Morse homology is defined through the equivariant Morse chain complex $C(f,g)$ commonly known as the equivariant Thom-Smale-Witten complex. 
Let $H \subset G$ be an arbitrary subgroup, then $M^H \subset M$ is a smooth submanifold. 
By restricting the Morse-Smale pair $(f,g)$ to $M^H$, we get a Morse-Smale pair $(f|_{M^H}, g|_{M^H})$, using which we can define a chain complex $C(f|_{M^H}, g|_{M^H})$ for $M^H$. A benefit of $f$ being stable is that $C(f|_{M^H}, g|_{M^H})$ is a subcomplex of $C(f,g)$.

On the other hand, if we start with an equivariant Morse function $f$ that is not stable,
and if we perturb $f$ to $f^\epsilon$, it is challenging to compute the chain complex $C(f^\epsilon,g)$. \cb  However, when $f^\epsilon$ is sufficiently close to $f$, smooth trajectories for $f^\epsilon$ can be obtained by gluing certain broken trajectories for $f$.

In this paper, we define an equivariant Thom-Smale-Witten complex for the pair $(f,g)$ by counting certain broken trajectories. To count these broken trajectories, we work in the case when the moduli spaces are not too bad. Note that given an unstable equivariant Morse function $f$, it is not always possible to find a metric $g$ such that the pair $(f,g)$ is Morse-Smale. For this reason, we introduce a weaker condition, \emph{stably Morse-Smale} (see Definition~\ref{def: stably morse smale}), which roughly means that ascending and descending manifolds intersect cleanly.

We pose the following question:
\begin{question}
Given an equivariant Morse function $f$, is the condition that $(f,g)$ is stably Morse-Smale a generic condition on $g$?
\end{question}

We give an affirmative answer (see Theorem~\ref{thm: stably metric}) to this question when $G$ is the group generated by a reflection $\phi$. We simply write $G = \langle \phi \rangle$.

\begin{definition}[Reflection]
    A diffeomorphism $\phi: M \to M$ is a reflection, if $\phi^2 = \op{id}$ and the fixed point set $\{x \in M ~|~ \phi(x) = x\}$ is a codimension-one submanifold of $M$. 
\end{definition}

When $G = \langle \phi \rangle$ and the pair $(f,g)$ is stably Morse-Smale, we define an equivariant Thom-Smale-Witten complex (see Theorem~\ref{thm: bold d square is zero}) by counting certain broken trajectories, and we use the obstruction bundle gluing developed in \cite{hutchings2009gluing} to show differential squared is zero. The resulting chain complex is canonical. Note that the approach developed in \cite{bao2024morse} works for all finite group actions, but this construction first stabilizes the Morse function $f$ to $f^\epsilon$, and the chain complex for the pair $(f^\epsilon, g)$ depends on the choices of stabilization. Note that this stabilization is only canonical for reflection actions.

The same construction applies in the case where $M$ is a compact manifold with boundary, since we can double $M$ and get a manifold with a reflection. This gives an alternative definition to Kronheimer and Mrowka's definition of Morse homology for manifolds with boundaries in Section~2 of \cite{kronheimer2007Monopoles}.

We also explore the case when $G$ is generated by multiple reflections, and compute an equivariant Thom-Smale-Witten complex of an upright higher-genus Riemann surface, where many gradient trajectories are not transversely cut out. We want to emphasize that the advantage of our definition is that we do not perturb the moduli spaces, and hence it is computable.

The paper is organized as follows:
\be
\item[-] Section 2 examines the genericity of stably Morse-Smale pairs.
\item[-] Section 3 presents our alternative definition of the Thom-Smale-Witten complex for manifolds with boundary and proves its quasi-isomorphism to the one in \cite{kronheimer2007Monopoles}.
\item[-] Section 4 develops the equivariant Thom-Smale-Witten complex for closed manifolds with reflection actions.
\item[-] Section 5 proves the well-definedness of the chain complexes from Sections 3 and 4.
\item[-] Section 6 explores generalizations to groups generated by multiple reflections and computes an example of an upright higher genus surface.
\ee 

\subsection{Relation to other works}

Obstruction bundle gluing is first developed by \cite{hutchings2009gluing} for gluing obstructed pseudo-holomorphic curves. 
For the Morse homology for manifolds with boundary, \cite{kronheimer2007Monopoles} also glues broken trajectories, but in that case, every broken trajectory glues. 
In the Morse setting, \cite{bao2024morse} extended this to the general case where descending and ascending manifolds intersect cleanly. 
Our work is a very special case of \cite{bao2024morse}, but we manage to write down explicitly which broken trajectories glue, and when they glue.
We acknowledge that Ipsita Datta and Yuan Yao are concurrently working on obstruction bundle gluing in the Morse case in \cite{datta2024obstruction}.

The construction of the chain complex after determining gluing behavior has been approached differently in various works. 
In \cite{bao2024computable}, the chain complex is generated by all critical points, and to define a differential, an abstract perturbation is carried out as in the standard Kuranishi setting, which is independent of the geometry of the ambient manifold and hence has the advantage of being able to generalize to the Floer case straightforwardly. However, there are many choices of perturbations, although up to homotopy there are finitely many. In \cite{kronheimer2007Monopoles}, the chain complex is generated by all interior critical points and stable boundary critical points, and the differential counts broken trajectories. In our approach (except Section 6), we include all critical points, and more than that, each unstable critical point $\cP{p}$ gives rise to three generators $\cP{p}$, $\cP{p}_+$, and $\cP{p}_-$, and the differential counts certain broken trajectories. The choices we made are the orientation of descending manifolds, and a nonzero normal vector for each critical point in the fixed point set.

\cbu
Seidel and Smith in \cite{seidel2010localization} provide two definitions of equivariant Morse homology when $G = \Z_2$: 
In Section 2a, they construct a Borel type of Equivariant Morse homology using a family of Morse function on $M$ parametrized by $BG$. In Section 2b, they give another definition of Equivariant Morse homology using a $U$-map, under the assumption that there exists an equivariant metric $g$ such that the pair $(f,g)$ is Morse-Smale and the assumption that for any critical point $\cP{p} \in M^G$, $\ind_{f}(\cP{p}) - \ind_{f_{M^G}}(\cP{p})$ is independent of $\cP{p}$. Note that in the case when $f$ is stable, this is satisfied and $\ind_{f}(\cP{p}) - \ind_{f_{M^G}}(\cP{p}) = 0$. Using their second definition, a Smith type of inequality is achieved. On the other hand, Smith inequality using Morse theory is achieved in general in \cite{bao2024morse}, and this paper is a special (reflections only), limiting case of \cite{bao2024morse}, so the Smith inequality follows directly from the definition in this paper (See Remark 9.5 in \cite{bao2024morse}). 
\cb

\section{Genericity of Stably Morse-Smale Pair}
In this section, we define the notion of a stably Morse-Smale pair. We show that when the group $G$ is generated by a reflection, the stably Morse-Smale condition is generic.

\subsection{Definition of Stably Morse-Smale for an Arbitrary Group}

Let $M$ be a smooth manifold, $G$ a finite group acting on $M$ via diffeomorphisms, $f: M \to \R$ a $G$-equivariant (or simply equivariant) Morse function, and $g$ an equivariant metric on $M$.

We first recall some preliminary properties of equivariant Morse functions and metrics. The proofs of the following two lemmas are Lemma~7.2 and Lemma~7.3 in \cite{bao2024morse}.

Let $H \subset G$ be a subgroup. 
\begin{lemma} \label{lemma: equivariant of tangent is the tangent of equivariant}
    For any $\cP{p} \in M^H$, the $H$-equivariant subspace $(T_\cP{p} M)^H$ is the same as $T_\cP{p} (M^H)$.
\end{lemma}

For this reason, we write $T_{\cP{p}}M^H$.

\begin{lemma}
    $\crit(f_{M^H}) = \crit(f) \cap M^H$.
\end{lemma}

\begin{lemma}\label{lemma: boundary implies interior}
Let $\cP{p}, \cP{q} \in \crit(f|_{M^H})$ be two critical points. Suppose that $\cP{q}$ is stable, and $\mathscr D_\cP{p} \cap M^H$ intersects $\mathscr A_\cP{q} \cap M^H$ transversely in $M^H$. Then $\mathscr D_\cP{p}$ intersects $\mathscr A_\cP{q}$ transversely in $M$.
\end{lemma}
\begin{proof}
    Let $\mathscr D = \mathscr D_\cP{p}$ and $\mathscr A = \mathscr A_\cP{q}$. Note that it suffices to show that $\mathscr D$ intersects $\mathscr A$ transversely in a small neighborhood of $\cP{q}$. Since $\cP{q}$ is stable, we have $\mathscr D_\cP{q} \subset M^{\stab(\cP{q})} \subset M^H$. Therefore, $T_\cP{q} M^H + T_\cP{q} \mathscr A = T_\cP{q} M$, and hence for \cbu $z \in M^H \cap \mathscr A$ \cb close to $\cP{q}$, $T_z M^H + T_z \mathscr A = T_z M$. By assumption, $T_z \mathscr D + T_z \mathscr A \supset T_z M^H$, and hence \cbu $T_z \mathscr D + T_z \mathscr A \supset T_z M$. \cb
\end{proof}

\begin{definition}[Clean Intersection]
    Two submanifolds $X$ and $Y$ of $Z$ are said to \emph{intersect cleanly} if $X \cap Y$ is a submanifold of $Z$, and for any $z \in X \cap Y$, $T_z(X \cap Y) = T_z X \cap T_z Y$.
\end{definition}

\begin{definition}[Stably Morse-Smale Pair]\label{def: stably morse smale}
    We say the pair $(f,g)$ is \emph{stably Morse-Smale} if for all $\cP{p}, \cP{q} \in \crit(f)$, the descending manifold $\mathscr D_\cP{p}$ intersects the ascending manifold $\mathscr A_\cP{q}$ \emph{cleanly}, and for any $z \in \mathscr D_{\cP{p}} \cap \mathscr A_{\cP{q}}$, 
    $$T_z \mathscr{D}_\cP{p}  + T_z \mathscr{A}_{\cP{q}}  \supset T_z M^{\stab(z)}.$$
\end{definition}
\begin{lemma}
    Suppose that the pair $(f,g)$ is stably Morse-Smale and that $f$ is stable. Then the pair $(f,g)$ is Morse-Smale.
\end{lemma} 

\begin{proof}
    For any $z \in \mathscr D_{\cP{p}} \cap \mathscr A_{\cP{q}}$, let $H = \stab(z)$.
    Since the initial condition of the equivariant ODE $$\dot \gamma + \op{grad}_{f,g}(\gamma) = 0, \quad \gamma(0) = z$$ is preserved by $H$, the limit of the solution is also preserved by $H$. In other words, $H \subset \stab(\cP{p})$. Since $f$ is stable, $\mathscr D_\cP{p} \subset M^{\stab(\cP{p})}$. Hence, $z \in M^{\stab(\cP{p})}$, which implies $\stab(\cP{p}) \subset H$. This means $\stab(z) = \stab(\cP{p})$.
    Then the stably Morse-Smale condition implies that $\mathscr D_\cP{p}$ intersects $\mathscr A_\cP{q} \cap M^H$ transversely in $M^H$. By Lemma~\ref{lemma: boundary implies interior}, $\mathscr D_\cP{p}$ intersects $\mathscr A_\cP{q}$ transversely in $M$.
\end{proof}

\subsection{Properties of Stably Morse-Smale for Reflections}
In this section, we study the stably Morse-Smale condition, specializing to the case when $G$ is generated by a reflection.

\begin{example}\label{example: double}
Let $N$ be a manifold with boundary. Then $M \coloneq N \cup_{\partial N} \overline{N}$ admits a reflection group action. 
\end{example}

Note that $M \backslash M^G$ is not necessarily disconnected, as seen in the example of the Klein bottle.

\begin{lemma}
    Suppose $G$ is generated by a reflection. Then for an equivariant Morse function $f$ and an equivariant metric $g$, the pair $(f,g)$ is stably Morse-Smale if and only if
    for any critical points $\cP{p}, \cP{q}$, if $\D_\cP{p}$ intersects $\A_\cP{q}$ non-transversely in $M$, then:
    \begin{enumerate}
        \item $\cP{p}, \cP{q} \in M^G$;
        \item $\cP{p}$ is stable and $\cP{q}$ is unstable;
        \item $\D_\cP{p} \pitchfork_{M^G} \A_{\cP{q}}$.
    \end{enumerate}
\end{lemma}

\begin{remark}
    This is the condition introduced in Definition 2.4.2 \cite{kronheimer2007Monopoles} for manifolds with boundary.
\end{remark}

\begin{proof}
    Suppose $(f,g)$ is stably Morse-Smale and $\mathscr D = \mathscr D_{\cP{p}}$ intersects $\mathscr A = \mathscr A_{\cP{q}}$ non-transversely. Then there exists $z \in \mathscr D \cap \mathscr A$ such that $T_z \mathscr D + T_z \mathscr A \neq T_z M$. Hence, $\stab(z) = G$. This implies $\cP{p}, \cP{q} \in M^G$ and $T_z \mathscr D + T_z \mathscr A = T_z M^G$. 
    By Lemma~\ref{lemma: boundary implies interior}, we get $\cP{q}$ is unstable. 
    Since $G$ is generated by a reflection, a vector $n \in T_\cP{p} M$ such that $n \perp T_\cP{p} M^G$ is an eigenvector of the Hessian of $f$ at $\cP{p}$. Thus, $n$ is either tangent to the ascending manifold of $\cP{p}$ or the descending manifold of $\cP{p}$. Replacing $f$ by $-f$ swaps descending manifolds and ascending manifolds.
    Suppose $\cP{p}$ is unstable, i.e., $n$ is tangent to $\mathscr D$.  With respect to $-f$, $\cP{p}$ is stable. We apply Lemma~\ref{lemma: boundary implies interior} for $-f$, and we get $\mathscr D$ and $\mathscr A$ intersect transversely in $M$, contradicting the non-transverse assumption. Therefore, $\cP{p}$ is stable. The other direction is obvious.
\end{proof}

\begin{lemma}\label{lemma: transversality interiro implies boundary}
    Suppose $G$ is generated by a reflection. \cbu If the pair $(f,g)$ is stably Morse-Smale, \cb the pair $(f,g)|_{M^G}$ is Morse-Smale.
\end{lemma}
\begin{proof}
    We only need to show that when critical points $\cP{p},\cP{q}\in M^G$ and when $\cP{p}$ is unstable and $\cP{q}$ is stable, then $(\mathscr D_\cP{p} \cap M^G) \pitchfork_{M^G} (\mathscr A_\cP{q} \cap M^G)$. For any $z \in \mathscr D_\cP{p} \cap  \mathscr A_\cP{q} \cap M^G$ and $w \in T_z(\mathscr D_\cP{p} \cap \mathscr A_\cP{q} \cap M^G)$, there exist $u \in T_z \mathscr D_\cP{p} $ and $v \in T_z \mathscr A_{\cP{q}}$ such that $u + v = w.$
    By averaging both sides, we complete the proof.
\end{proof}

\begin{theorem}\label{thm: stably metric}
    Let $k \geq 1$ be an integer. Let $f$ be an equivariant Morse function of class $C^{k+1}$. Suppose that $G$ is generated by a reflection of $M$. For a generic equivariant metric $g$ of class $C^k$, the pair $(f,g)$ is stably Morse-Smale.   
\end{theorem}

\begin{proof}
    Recall that the stably Morse-Smale condition is an open condition with respect to the space of equivariant metrics. Then it suffices to show that the subspace of the equivariant metrics $g$ such that $(f,g)$ is stably Morse-Smale is dense. 

    Given any equivariant $g$, we first choose a metric $g'$ on $M^G$ such that $g'$ is $C^k$-close to $g|_{M^G}$ and the pair $(f|_{M^G}, g')$ is Morse-Smale. We then extend $g' - g$ to a small neighborhood of $M^G$ trivially, multiply it by a bump function, and then average it via $G$. This gives a small equivariant symmetric tensor $\delta$ on $M$ supported in a small neighborhood of $M^G$. We define a metric $g'$ on $M$ by $g' = g + \delta$.

    Next, around each critical point outside $M^G$, we perturb the metric $g'$ into an equivariant $g''$ such that $\mathscr D_{\cP{p}}$ intersects $\mathscr A_{\cP{q}}$ transversely in $M$ if at least one of $\cP{p}$ and $\cP{q}$ is outside $M^G$. The details of this perturbation can be found in the proof of Theorem 1.8 in \cite{bao2024morse}. The perturbation is supported away from $M^G$, so $(f, g'')|_{M^G}$ is still Morse-Smale.

    Now, for the case $\cP{p}, \cP{q} \in \crit(f)$ where both $\cP{p}$ and $\cP{q}$ are stable, by Lemma~\ref{lemma: boundary implies interior}, we have $\mathscr D_\cP{p}$ intersects $\mathscr A_\cP{q}$ transversely. The case when $\cP{p}, \cP{q} \in \crit(f)$ and both $\cP{p}$ and $\cP{q}$ are unstable can be handled similarly by replacing $f$ with $-f$.

    The only remaining case is when $\cP{p}$ is unstable, $\cP{q}$ is stable, \cbu and $p,q \in M^G$. \cb Let $n \in T_\cP{q} M$ be a tangent vector perpendicular to $M^G$. Then $n \in T_\cP{q} \mathscr A_\cP{q}$. Note that for any $z \in \mathscr D_\cP{p} \cap \mathscr A_\cP{q} \cap M^G$ we have $T_z \mathscr D_\cP{p} + T_z \mathscr A_\cP{q} \supset T_z M^G$. But since $n \in T_\cP{q} \mathscr A_\cP{q}$, we have for all $z \in \mathscr D_\cP{p} \cap \mathscr A_\cP{q}$ close to $\cP{q}$, 
    \begin{equation}\label{eqn: cherry}
        T_z \mathscr D_\cP{p} + T_z \mathscr A_\cP{q} = T_z M.
    \end{equation} 
    Since the transversality property is preserved by the gradient flow, Equation~\ref{eqn: cherry} holds for all $z \in \mathscr D_\cP{p} \cap \mathscr A_\cP{q}$. 
\end{proof}

\subsection{Stabilization}
Given any equivariant Morse function $f$ and any $\epsilon > 0$, \cbu by Theorem 1.7 of \cite{bao2024morse} or Theorem 1.7 of \cb \cite{mayer1989Ginvariant} there exists an equivariant stable Morse function $f^\epsilon$ such that $|| f - f^\epsilon ||_{C^0} < \epsilon$. In the case of a reflection action, $f^\epsilon$ is constructed by perturbing $f$ near unstable critical points such that: 
\be 
    \item Each unstable critical point $\cP{p}$ of $f$ remains a critical point of $f^\epsilon$ with $\ind_{f^\epsilon} (\cP{p}) = \ind_f(\cP{p}) - 1$.
    \item In a small neighborhood of each unstable critical point $\cP{p}$ of $f$, there exist two additional critical points $\cP{p}_+$ and $\cP{p}_-$ of $f^\epsilon$ on opposite sides of $M^G$, satisfying $\ind_{f^\epsilon} (\cP{p}_\pm) = \ind_f(\cP{p})$.
\ee 

We ask the following question:
\begin{question}\label{question: name of stably ms}
    Suppose $G$ is generated by a reflection.
    For a stably Morse-Smale pair $(f,g)$, if $\epsilon$ is sufficiently small, is $(f^\epsilon, g)$ Morse-Smale?
\end{question}

\section{Morse homology for manifolds with boundary}
\cbu In this section, we review Kronheimer and Mrowka's definition of Morse homology for manifolds with boundary, and then we give an alternative definition, which is geometric and also leads to a geometric description of Kronheimer and Mrowka's definition. This section can be viewed as a special case of Section~\ref{section: equivariant chain complex}. \cb

Let $N$ be a manifold with boundary $\p N$. 
In this section, we first review the definition of chain complexes for $N$ in \cite{kronheimer2007Monopoles} and give an alternative definition of chain complexes for manifolds with boundary. We also show that the two chain complexes are quasi-isomorphic.

Given $N$, we can construct its double $M = N \cup_{\p N} \overline{N}$. Let $G = \langle \phi \rangle$ be the group generated by the obvious reflection $\phi: M \to M$. Then $M$ is a smooth manifold with a $G$-action. We consider a $G$-equivariant Morse function $f: M \to \R$ and a $G$-equivariant metric $g$ on $M$. We write $f_N$ to mean $f\vert_{N}.$

Let $(f,g)$ be a stably Morse-Smale pair for $M$ with respect to the $G$-action. 
\begin{definition}
    A critical point $\cP{p}$ of $f_N$ is said to be: 
    \be
        \item an \emph{interior critical point} if $\cP{p}$ lies in the interior $\op{int}(N)$ of $N$ and we denote this by $\cP{p} \in \crit^o(f_N)$;
        \item a \emph{boundary critical point} if $\cP{p} \in \p N$;
        \item a \emph{stable (unstable) boundary critical point} if $\cP{p} \in \p N$ and is stable (unstable) and we denote this by $\cP{p} \in \crit^s(f_N)$ ($\cP{p} \in \crit^u(f_N)$).
    \ee 
\end{definition}
From this definition, we have that the critical point set $\crit(f_N)$ of $f_N$ is given by $\crit^o(f_N) \coprod \crit^s(f_N) \coprod \crit^u(f_N)$.

\subsection{Orientation of moduli spaces}\label{section: orientation of moduli spaces}
To orient the moduli spaces, we orient the descending manifolds of all critical points. 
First, we orient the descending manifolds in $\p N$ for the pair $(f_{\p N},g_{\p N})$.
Then, we orient descending manifolds in $M$:
\be 
    \item For any critical point $\cP{p} \in \crit^o(f_N)$, we choose an arbitrary orientation of $\mathscr D_\cP{p} $.
    \item For any critical point $\cP{p} \in \crit^s(f_N)$, we orient $\mathscr D_\cP{p} $ using the orientation chosen in \cbu $\partial N$. \cb
    \item For any critical point $\cP{p} \in \crit^u(f_N)$, we choose the orientation of $T_\cP{p} \mathscr D_\cP{p} $ such that the isomorphism $T_\cP{p} \mathscr D_\cP{p} \simeq \R\langle \mathfrak o_\cP{p} \rangle \oplus T_\cP{p} (\mathscr D_\cP{p} \cap \p N)$ is orientation-preserving, where $\mathfrak o_{\cP{p}}$ is the outward-pointing normal to $\p N$, and the orientation of $T_\cP{p} (\mathscr D_\cP{p} \cap \p N)$ is chosen as above.
\ee 
With these choices, for each rigid trajectory that appears below, we can assign $\pm 1$.
\cbu Note that the same orientation is adapted in Section~\ref{section: equivariant chain complex} except that there is no canonical choice of $\mathfrak o_\cP{p}$. \cb

\subsection{Kronheimer and Mrowka's chain complexes}\label{section: km}
Let $R$ be a commutative ring.
For any integer $k \geq 0$, and for $\dagger \in \{o, s, u\}$, let $C_k^\dagger$ be the free $R$-module generated by the critical points $\crit^\dagger(f_N)$ with $\ind_{f_N}(\cP{p}) = k$, and let $C^\dagger = \oplus_k C_k^\dagger$.

For all $\star, \dagger \in \{o, s, u\}$ except for the combination $(\star, \dagger) = (s, u)$, let $$\partial_\dagger^\star: C^\star \to C^\dagger$$ be the linear map defined by counting rigid Morse trajectories inside $M$.
Similarly, let $(\bar{C}, \bar{\p})$ be the Thom-Smale-Witten complex of $f_{\p N}$, graded by the Morse index of $f_{\p N}$.

For $\star, \dagger \in \{s,u\}$, let $$\bar\partial_\dagger^\star: C^\star \to C^\dagger$$ be the linear map defined by counting rigid Morse trajectories inside $\p N$. With respect to the splitting $\bar C = C^s \oplus C^u[-1]$, we write 
\[
\bar\p = 
\begin{bmatrix}
     \bar\partial_s^s &  \bar \partial_s^u  \\
    \bar \partial_u^s & \bar \partial_u^u \\
\end{bmatrix},
\] \cbu 
where the grading of $C^u[-1]$ is defined by $|\cP{q}| = \ind_{f_N} (\cP{q}) - 1$, for any $\cP{q} \in \op{Crit}^u(f_N)$. \cb

Kronheimer and Mrowka defined the two chain complexes $(\check{C}, \check \p)$ and $(\hat{C}, \hat \p)$  as follows:
\[
\check C = C^o \oplus C^s, 
\]
and 
\begin{equation*}
    \check{\partial} =
    \begin{bmatrix}
        \partial_{o}^{o} & -\partial_{o}^{u} \bar{\partial}_{u}^{s} \\
        \partial_{s}^{o} & \bar{\partial}_{s}^{s}-\partial_{s}^{u} \bar{\partial}_{u}^{s}
    \end{bmatrix};
\end{equation*} 
\[
    \hat C = C^o \oplus C^u,
\]
and 
\begin{equation*}
        \hat{\partial}  =
        \begin{bmatrix}
            \partial_{o}^{o} & \partial_{o}^{u} \\
            -\bar{\partial}_{u}^{s} \partial_{s}^{o} & -\bar{\partial}_{u}^{u}-\bar{\partial}_{u}^{s} \partial_{s}^{u}
        \end{bmatrix}.
\end{equation*}

\begin{theorem}[\cbu Theorem 2.4.5 of \cb \cite{kronheimer2007Monopoles}]
    The operators $\check{\partial}$ and $\hat{\partial}$ both square to zero. The complexes $(\check{C}, \check{\partial})$ and $(\hat{C}, \hat{\partial})$ compute the absolute and relative homology groups, respectively, with the isomorphisms
    $$
    \begin{aligned}
    & H_{k}(\check{C}, \check{\partial}) \simeq H_{k}(N) \\
    & H_{k}(\hat{C}, \hat{\partial}) \simeq H_{k}(N, \p N).
    \end{aligned}
    $$ 
\end{theorem} 

\subsection{An alternative definition via stabilization}
Now we define a larger chain complex $(\check{\mathbb C}_*,\check{\fatD})$ by 
\[
    \check{\mathbb C} =  C^o \oplus C^s \oplus C^u[-1] \oplus C^u_-,
\] 
where $C^u_{-}$ is generated by $\cP{p}_{-}$ for every $\cP{p} \in \crit^u(f\vert_N)$ and can be viewed as $C^u$. We define the differential
\[
\check{\fatD} = 
\begin{bmatrix}
    \partial_o^o & 0 & 0 & \partial_o^u\\
    \partial_s^o & \bar\partial_s^s &  \bar \partial_s^u & \partial_s^u \\
    0 & \bar \partial_u^s & \bar \partial_u^u  & \Pi_- \\
    -\bar \partial_u^s \partial_s^o & 0 & 0 & -\bar \partial_u^u - \bar \partial_u^s \partial_s^u
\end{bmatrix},
\]
where $\Pi_-: C^u_- \to C^u[-1]$ maps $\cP{p}_- \mapsto \cP{p}$.

\begin{remark}
    This definition is motivated by the stabilization process. Specifically, we modify the Morse function $f_N$ near each unstable critical point $\cP{p}$ to obtain a new Morse function $f^\epsilon_N$ such that an extra critical point $\cP{p}_-$ is created. This forms the last two summands in $\check{\mathbb C}$    with $\ind_{f^\epsilon_N}(\cP{p}) = \ind_{f_N}(\cP{p}) - 1$ and $\ind_{f^\epsilon_N}(\cP{p}_-) = \ind_{f_N}(\cP{p})$.
\end{remark}

\begin{proposition}\label{prop: fatD square is zero}
    $\check{\fatD}^2 = 0$.
\end{proposition}
The proof of this is postponed to Section~\ref{section: d square is zero}.

It is clear that $(\bar C, \bar \p)$ is a subcomplex of $\check{\mathbb C}$. The complex $(\hat{\mathbb C}, \hat{\fatD})$ is defined to be the quotient complex $(\check{\mathbb C},\check{\fatD})/(\bar C, \bar \p)$.

\begin{proposition}
    The chain complex $(\check{\mathbb C}, \check{\fatD})$ is quasi-isomorphic to $(\check C, \check{\p})$,
    and $(\hat{\mathbb C}, \hat{\fatD})$ is quasi-isomorphic to $(\hat C, \hat \p)$.
\end{proposition}
\begin{proof}
    We first show that $(\check{\C}, \check \fatD)$ is quasi-isomorphic to $(\check C, \check \p)$.

    We define the map $\psi: (\check{\C}, \check \fatD) \to (\check C, \check \p)$
    by
    $$\psi = \begin{bmatrix}
        \op{id} & 0 & -{\partial}^u_o & 0\\
        0 & \op{id} & -{\partial}^u_s & 0\\        
    \end{bmatrix}.$$
    
It suffices to show the claim below.
\begin{claim}
    The map $\psi$ is a chain map.
\end{claim}
\begin{proof}
    This is a straightforward calculation using the identities in $\check{\fatD}^2 = 0.$
\end{proof}

Now we calculate the homology of $(\ker \psi, \check{\fatD}|_{\ker \psi})$.
It is easy to see that $$\ker \psi = \Biggl\{  \begin{bmatrix}
    \p^u_o x \\
    \p^u_s x \\
    x \\
    y
\end{bmatrix}
~\Bigg|~ x \in C^u[-1], y\in C^u_-
\Biggr\}.$$

Elements in $\ker \check{\fatD}|_{\ker \psi}$ are given by $\begin{bmatrix}
    \p^u_o x \\
    \p^u_s x \\
    x \\
    -(\bar\p^s_u \p^u_s + \bar\p^u_u) x
\end{bmatrix}$, for $x \in C^u[-1]$.
Using Proposition~\ref{prop: fatD square is zero} it is easy to see:
\[
\begin{bmatrix}
    \p^u_o x \\
    \p^u_s x \\
    x \\
    -(\bar\p^s_u \p^u_s + \bar\p^u_u) x
\end{bmatrix} = \check{\fatD} \begin{bmatrix}
        \p^u_o x \\
     \p^u_s x \\
    x \\
    (-\bar\p^s_u \p^u_s - \bar\p^u_u ) x + x_-
\end{bmatrix},
\]
and hence they are all exact, and the homology of $(\ker \psi, \ker \check{\fatD}|_{\ker \psi})$ vanishes.
The long exact sequence associated with the short exact sequence 
$$0 \to \ker \psi  \to \check{\C} \to \check{C} \to 0$$
implies that $\psi$ is a quasi-isomorphism.

Now we consider the map $\phi: \check{\C} \to \hat C$ defined by sending $C^o$ to $C^o$ via the identity map, $C^s \oplus C^u[-1]$ to $0$, and $C^u_-$ to $C^u[-1]$ via $\Pi_-$. 
Clearly, $\phi$ is surjective, and $\ker \phi = \bar C$. One can check that $\phi$ is a chain map. Therefore, $\phi$ induces an isomorphism from $(\hat{\C}, \hat{\fatD})$ to $(\hat C, \hat \p)$.
\end{proof}

\subsection{A geometric explanation}
In this section, we provide alternative descriptions of Kronheimer and Mrowka's chain complexes to complete the geometric picture. The arguments in this section rely on gluing results that will be developed in an upcoming paper, so the claims are conjectural.

Recall $N$ is a manifold with boundary.
Given a stably Morse-Smale pair $(f,g)$ on $N$, and a small $\epsilon > 0$, following \cite{bao2024morse}, we construct a stable Morse function $f^\epsilon$. 
Suppose the answer to Question~\ref{question: name of stably ms} is positive. Then the pair $(f^\epsilon, g)$ is Morse-Smale. 
We consider the Thom-Smale-Witten complex of $(f^\epsilon, g)$, which is conjectured to be quasi-isomorphic to the chain complex $(\check{\C}, \check{\fatD})$.

In this case, one can define a Thom-Smale-Witten complex $(C^o_{\epsilon}, \partial^o_\epsilon)$ using only the interior critical points of $f^\epsilon$. Note that since $f^\epsilon$ is stable, there can be no broken trajectories going from the interior to the boundary and returning. Therefore, we have $(\partial^o_\epsilon)^2 = 0$. We expect the homology of $(C^o_{\epsilon}, \partial^o_\epsilon)$ to be $H_*(N,\p N)$.
On the other hand, on the level of generators, for each critical point $\cP{p} \in \crit^s(f)$, it does not appear in $\crit^o(f^\epsilon)$; for each critical point $\cP{p} \in \crit^u(f)$, it corresponds to two critical points $\cP{p} \in \crit^s(f^\epsilon)$ and $\cP{p}_- \in \crit^o(f^\epsilon)$. The critical point $\cP{p}$ is not included in the generators of $C^o_\epsilon$, but $\cP{p}_-$ is. This gives intuition about the chain complex $(\hat C, \hat \p)$, noting that $\hat C =  C^o \oplus C^u$. To compare the differentials $\check \p$ and $\p_\epsilon^o$, one needs some gluing results which are yet to be developed.

To understand the chain complex $(\check{C}, \check{\p})$, we can construct a de-stabilization. We denote $f^{-\epsilon} \coloneqq -(-f)^\epsilon$, i.e., the opposite of the stabilization of $-f$.
This has the effect of preserving all the unstable critical points of $f$, while modifying $f$ near each stable critical point:
each stable critical point $\cP{p} \in \crit^s(f)$ corresponds to two critical points $\cP{p} \in \crit^u(f^{-\epsilon})$ and $\cP{p}_- \in \crit^o(f^{-\epsilon})$. We can construct a chain complex $(C^o_{-\epsilon}, \p^o_{-\epsilon})$ by using only the interior critical points of $f^{-\epsilon}$. It is clear that $(\p^o_{-\epsilon})^2 = 0$, since there does not exist a broken trajectory going from the interior to the boundary and returning. By a similar argument as before, we expect that $(C^o_{-\epsilon}, \p^o_{-\epsilon})$ is isomorphic to $(\check{C}, \check{\p})$, whose homology computes $H_*(N)$.

\section{Equivariant Morse homology}
\subsection{Equivariant Thom-Smale-Witten complex}\label{section: equivariant chain complex}
Let $M$ be a closed manifold, $\phi$ be a reflection on $M$, $G = \langle \phi \rangle$, and $(f,g)$ be a stably Morse-Smale pair. In this section we define an equivariant Thom-Smale-Witten complex for $(M, G)$.

Let $\mathcal N$ be the $1$-dimensional orthogonal complement of $T{M^G}$ in $TM|_{M^G}$. At any critical point $\cP{p}\in M^G$, it is easy to check that $\mathcal N_\cP{p}$ is an eigenspace of $\hess{f}(\cP{p})$.
We choose a unit-length vector $\mathfrak o_\cP{p} \in T_\cP{p} \mathcal N$ for the orientation of $\D_\cP{p}$. 

To orient the moduli spaces of trajectories, we orient all the descending manifolds in the same way as in Section~\ref{section: orientation of moduli spaces}, with the outward-pointing normal vector replaced by the unit vector $\mathfrak{o}_\cP{p}$ at each critical point $\cP{p}$. Note that the choice of orientation does not have to be $G$-equivariant.

Recall that we denote by $\crit^o(f)$ the set of interior critical points of $f$, $\crit^s(f)$ the set of stable critical points of $f$, and $\crit^u(f)$ the set of unstable critical points of $f$. For $\dagger \in \{o, s, u\}$, we write $C^\dagger$ for the chain complex generated by $\crit^\dagger(f)$ and graded by $\ind = \ind_f$, the Morse index of the critical point of $f$.
We define the chain complex
 \[
 \boldsymbol{C} = C^o \oplus C^s \oplus C^u[-1] \oplus C^u_+ \oplus C^u_-, 
 \]
where $C^u[-1]$ is as before and $\cP{p}_{\pm}$ are generators of $C^u_\pm$ for any $\cP{p}\in \crit^u(f)$.
These $\cP{p}_\pm$ correspond to the newly added critical points after stabilization, and the label $\pm$ is determined by the choice of $\mathfrak o_\cP{p}$.

 Next, we explain the differential $\boldsymbol{\p}$. We start by describing the maps for flowlines $\gamma \in \M{M}{p}{q}$ first.

Let $\cP{p}\in \crit^o(f)$, and $\cP{q}\in \crit^s(f)$ such that $\ind(\cP{p})-\ind(\cP{q}) = 1$. For any $\gamma \in \M{M}{p}{q}$, when $\vt$ is sufficiently large, we can write $\gamma(\vt) = \op{exp}_{\eta(\vt)} n(\vt)$, where $\eta(\vt) \in M^G$ and $n(\vt) \in \mathcal N_{\eta(\vt)}$, and $\op{exp}$ is the exponential map from a neighborhood of the zero section of $\mathcal N$ to $M$. Since $\op{grad}_{f}$ is tangent to $M^G$, we get $n(\vt) \neq 0$, and hence $$\gamma^\infty\coloneqq\langle \lim_{\vt\to \infty} n(\vt)/|n(\vt)|, \mathfrak o_\cP{q} \rangle \in \{1, -1\}.$$ 
Depending on the sign of $\gamma^\infty$, we write 
\[
\M{M}{p}{q} = \M{M}{p}{q_+} \coprod \M{M}{p}{q_-}
\] to be the disjoint union of flowlines from $\cP{p}$ that approach $\cP{q}$ from the $\gamma^\infty$ direction.
We define the linear maps $$\p^o_{s\pm}: C^o \to C^s$$ by counting elements of $\M{M}{p}{q_\pm}$.

In the same way, for any $\cP{p}\in \crit^u(f)$ and $\cP{q}\in \crit^o(f)$, such that $\ind(\cP{p})-\ind(\cP{q}) = 1$, 
we define 
$$\gamma^{-\infty}\coloneqq \langle \lim_{\vt\to -\infty} n(\vt)/|n(\vt)|, \mathfrak o_\cP{p} \rangle \neq 0.$$ 
Similarly, we write 
$$\M{M}{p}{q} = \M{M}{p_+}{q} \coprod \M{M}{p_-}{q}$$ to be the disjoint union of flowlines that start from the $\gamma^{-\infty}$ side of $\cP{p}$ to $\cP{q}.$
We define the linear maps $$\p_o^{u\pm}: C^u_\pm \to C^o$$ by counting elements of $\M{M}{p_\pm}{q}$.

Similarly, for $\cP{p}\in \crit^u(f)$ and $\cP{q}\in \crit^s(f)$ such that $\ind(\cP{p})-\ind(\cP{q}) = 1$, for any $\gamma \in \M{M}{p}{q}$, since $(f,g)$ is stably Morse-Smale, we have $\M{M^G}{p}{q}$ is transversely cut out by Lemma~\ref{lemma: transversality interiro implies boundary} and $\dim \M{M^G}{p}{q} = \ind(\cP{p}) -1 - \ind(\cP{q}) - 1 = -1$. Hence, $\M{M^G}{p}{q} = \emptyset$ and $\gamma$ is not contained in $M^G$. Then, by the same argument as above, we can define 
$$\gamma^{-\infty}\coloneqq \langle \lim_{\vt\to -\infty} n(\vt)/|n(\vt)|, \mathfrak o_\cP{p} \rangle \in \{1, -1\}$$ 
and 
$$\gamma^{\infty}\coloneqq \langle \lim_{\vt\to \infty} n(\vt)/|n(\vt)|, \mathfrak o_\cP{q} \rangle \in \{1, -1\}.$$
For $\diamond, \star \in \{+, -\}$, we define $\lm(M; \cP{p}_\diamond,\cP{q}_\star)$ analogously to the previous moduli spaces. If $\gamma^{-\infty}$ is $\diamond$ and $\gamma^\infty$ is $\star$, we also define four maps $$\p^{u\diamond}_{s\star}: C^u_{\diamond} \to C^s$$ 
by counting elements in $\lm(M; \cP{p}_\diamond,\cP{q}_\star)$.
We have 
\[
\M{M}{p}{q} = \coprod_{\diamond, \star \in \{+, -\}} \lm(M; \cP{p}_\diamond,\cP{q}_\star).
\]

Now we describe the flowlines $\gamma \in \M{M^G}{p}{q}$ in the fixed point set $M^G$.
For any $\cP{p} \in \crit^\dagger(f)$, $\cP{q} \in \crit^u(f)$ such that $\ind_{f_{M^G}}(\cP{p}) - \ind_{f_{M^G}}(\cP{q}) = 1$ and $\dagger \in \{s, u\}$, we consider an element $\gamma \in \M{M^G}{p}{q}$.

Consider a vector field $\xi$ along $\gamma$ such that 
\begin{equation}\label{eqn: kernel of D*}
- \nabla_\vt \xi + \nabla_\xi \op{grad}_f = 0
\end{equation}
and 
\begin{equation}\label{eqn: kernel of D* initial condition}
    \lim_{\vt\to -\infty} \frac{\xi}{|\xi|} = \mathfrak o_\cP{p}. 
\end{equation}
\begin{lemma}\label{lemma: kernel of D* is normal}
    $\xi$ is a section of $\mathcal N$ along $\gamma$.
\end{lemma}
\begin{proof}
    Let $\mathfrak n$ be the parallel transport of $\mathfrak o_\cP{p}$ along $\gamma$. Let $\phi\in G$ be the reflection. Then $\phi_* \mathfrak n + \mathfrak n$ is $G$-invariant, and hence it is tangent to $M^G$. But $\phi_* \mathfrak o_\cP{p} + \mathfrak o_\cP{p} = 0$. This implies that $\mathfrak n + \phi_* \mathfrak n = 0$, in other words, at any $\vt\in \R$, for any $v \in T_{\gamma(\vt)}M^G$, we have $\langle \mathfrak n, v \rangle = \langle \phi_* \mathfrak n, \phi_* v \rangle = \langle -\mathfrak n, v \rangle$, and hence $\langle \mathfrak n, v \rangle = 0$. Therefore, $\mathfrak n$ is a section of $\mathcal N$. For any $\vt\in \R$, let $\{e_1(\vt), \dots, e_n(\vt)\}$ be an orthonormal basis of $T_{\gamma(\vt)} M$ such that $\nabla_\vt e_i = 0$ for $i = 1, \dots, n$ and $e_1 =\mathfrak n$. Let $\xi = \sum_{i=1}^n\xi_i e_i$, then Equation~\ref{eqn: kernel of D*} becomes 
    \[
    - \frac{d}{d\vt} \xi_i + \sum_{j = 1}^n a_{ij}\xi_j = 0,
    \]
    where $a_{ij} = \langle e_i, \nabla_{e_j} \op{grad}_f\rangle$ is symmetric in $i,j$. Now we show that for $j\neq 1$, $a_{1j} = 0$. Note that $e_j$ is orthogonal to $\mathfrak n$ since they are orthogonal at $\cP{p}$. This implies that $e_j$ is tangent to $M^G$, and hence $G$-equivariant. But $\op{grad}_f$ is also $G$-equivariant. Therefore, $\nabla_{e_j} \op{grad}_f$ is also $G$-equivariant, and hence tangent to $M^G$. This shows $a_{1j} = a_{j1} = 0$.
\end{proof}

\begin{definition}
    We say  $\gamma \in \M{M^G}{p}{q}$ is \emph{orientation-preserving} (\emph{orientation-reversing}) if 
    \[
    \langle \lim_{\vt\to \infty}\frac{\xi}{|\xi|}, \mathfrak o_\cP{q} \rangle  > (<) 0.
    \]
    We denote the orientation-preserving subspace by $\mathcal M^+ (M^G; \cP{p}, \cP{q})$ and the orientation-reversing subspace by $\mathcal M^- (M^G; \cP{p}, \cP{q})$.
\end{definition}
\begin{remark}
Solutions of Equation~\ref{eqn: kernel of D*} and Equation~\ref{eqn: kernel of D* initial condition} exist and are unique up to a positive scalar multiplication. So the definition is independent of the choice of $\xi$.
\end{remark}

Then for any $\dagger \in \{s, u\}$, we define the orientation-preserving maps
$$P^\dagger_u: C^\dagger \to C^u$$ by counting orientation-preserving trajectories, and the orientation-reversing maps $$R^\dagger_u: C^\dagger \to C^u$$ by counting orientation-reversing trajectories. In particular, we have $\bar{\p}^\dagger_u = P^\dagger_u + R^\dagger_u$, and \[\M{M^G}{p}{q}=\mathcal M^+ (M^G; \cP{p}, \cP{q}) \coprod \mathcal M^{-} (M^G; \cP{p}, \cP{q})\] for $\dagger \in \{u,s\}.$  
Then the differential from $\boldsymbol{C}$ to $\boldsymbol{C}$ is defined by
$$\boldsymbol{\p} = 
\begin{bmatrix}
    \p_o^o & 0 & 0 & \p_o^{u+} & \p_o^{u-}\\
    \p_s^o & \bar\p_s^s & \bar \p_s^u & \p_s^{u+} & \p_s^{u-} \\
    0 & \bar \p_u^s & \bar \p_u^u  & -\Pi_+ & \Pi_-\\
    P^s_u \p^o_{s+} + R^s_u\p^o_{s-} & 0 & 0 & -P_u^u + P^s_u \p^{u+}_{s+} + R^s_u \p^{u+}_{s-} & R_u^u + P^s_u \p^{u-}_{s+} + R^s_u \p^{u-}_{s-}\\
    - P^s_u \p^o_{s-} - R^s_u\p^o_{s+} &  0 & 0 & R_u^u - P^s_u \p^{u+}_{s-} - R^s_u \p^{u+}_{s+} & -P_u^u - P^s_u \p^{u-}_{s-} - R^s_u \p^{u-}_{s+}
\end{bmatrix},$$
where $\Pi_+: C^u_{+} \to C^u[-1]$ maps $\cP{p}_+ \mapsto \cP{p}$, and $\Pi_-: C^u_{-} \to C^u[-1]$ maps $\cP{p}_- \mapsto \cP{p}$.
Note that the remaining terms in $\boldsymbol{\p}$ are all defined in sections \ref{section: km} above.

Here we provide a table to summarize the linear maps defined above.

\begin{center}
\begin{tabular}{| c | c |}
\hline
\text{Linear maps} & \text{Moduli spaces of trajectories} \\
\hline
$\p_{s_{\pm}}^{o}: C^o \to C^s$ & $\M{M}{p}{q_{\pm}}$\\
\hline
$\p_o^{u\pm}: C^u_{\pm} \to C^o$ & $\M{M}{p_\pm}{q}$\\
\hline
 $\p_{s\pm}^{u+}: C^u_{+} \to C^s$ &  $\M{M}{p_+}{q_\pm}$\\
\hline
$\p_{s\pm}^{u-}: C^u_{-} \to C^s$ &$\M{M}{p_-}{q_\pm}$\\
\hline
$P^\dagger_u: C^\dagger \to C^u, \dagger \in \{s, u\}$ &  $\mathcal M^+ (M^G; \cP{p}, \cP{q})$\\
\hline
$R^\dagger_u: C^\dagger \to C^u, \dagger \in \{s, u\}$ &  $\mathcal M^- (M^G; \cP{p}, \cP{q})$\\
\hline
\end{tabular}
\end{center}

\begin{theorem}\label{thm: bold d square is zero}
 $\boldsymbol{\p}^2 = 0$.
\end{theorem}

We postpone the proof to Section~\ref{section: d square is zero}. 

Let $(\bar C, \bar \p)$ be a chain complex on $M^G$, i.e., $\bar C = C^s \oplus C^u[-1]$ and 
\[
\bar\p = 
\begin{bmatrix}
     \bar\partial_s^s &  \bar \partial_s^u  \\
    \bar \partial_u^s & \bar \partial_u^u \\
\end{bmatrix}.
\]
Then the following is clear.
\begin{lemma}
The chain complex $(\bar C, \bar \p)$ is a subcomplex of $(\boldsymbol{C}, \boldsymbol{\p})$.
\end{lemma}

Now we explain the $G$-action on the chain complex $(\boldsymbol{C}, \boldsymbol{\p})$. This part can be viewed as a special case of section 3.6 in \cite{bao2021equivariant}.

For any $\cP{p} \in \crit(f)$ and $h \in G$, the map $dh: T_\cP{p} \mathscr D_\cP{p} \to T_{h \cP{p}} \mathscr D_{h\cP{p}}$ is orientation-preserving or reversing. We define $\sigma(h, \cP{p})=\pm 1$ if $dh$ is respectively orientation-preserving or reversing.
We then define a $G$-action on $\boldsymbol{C}$ to be the obvious action multiplied by $\sigma$. To be precise, let $G = \{e, \phi\}$. Then $e$ acts as the identity on $\boldsymbol{C}$, and $\phi$ acts as follows:

\begin{enumerate}

    \item For each generator $[\cP{p}]$, $\phi$ maps $C^o$, $C^s$, and $C^u[-1]$ to themselves and 

    $$\phi[\cP{p}] = \sigma(\phi, \cP{p}) [\phi \cP{p}].$$ In particular, 

    \begin{enumerate}

    \item if $[\cP{p}] \in C^o$ and if we choose the orientation of $\mathscr D_\cP{p}$ and $\mathscr D_{\phi \cP{p}}$ so that $d\phi$ preserves the orientation, then $\phi[\cP{p}] = [\phi \cP{p}]$;
  
     \item if $[\cP{p}] \in C^s$ and if we choose the orientation of $\mathscr D_\cP{p}$ and $\mathscr D_{\phi \cP{p}}$ so that $d\phi$ preserves the orientation, then $\phi[\cP{p}] = [\phi \cP{p}]$.
       
    \item if $[\cP{p}] \in C^u[-1]$, then $\phi[\cP{p}] = [\cP{p}]$.

    \end{enumerate}

    \item For each generator $[\cP{p}]$, $\phi: C^u_\pm \to C^u_\mp$ is given by: 

    $$\phi [\cP{p}_\pm] = \sigma(\phi, \cP{p}) [\cP{p}_\mp].$$ In particular, $\phi[\cP{p}_\pm] = - [\cP{p}_\mp]$.

\end{enumerate}

\begin{lemma}

    The $G$-action on $\boldsymbol{C}$ commutes with $\boldsymbol{\p}$.

\end{lemma}

\begin{proof}

    We prove $\phi \p_o^o = \p_o^o \phi$ and leave the rest for the reader to check. For any $\cP{p}, \cP{q} \in \crit^o(f)$ with $\ind (\cP{p}) - \ind(\cP{q})= 1$ and any $[\gamma] \in \M{M}{p}{q}$, we define the orientation of $\gamma$ as in standard Morse homology. Specifically, we define $\orientation{\gamma} \in \{1,-1\}$ such that the following isomorphism   \begin{equation}
        T_\cP{p} \mathscr D_\cP{p}  \simeq \orientation{\gamma} \mathbb R \langle \partial_\vt \rangle \oplus T_\cP{q} \mathscr D_\cP{q}
    \end{equation}
    induced by the linearized negative gradient flow is orientation-preserving and
 $s$ represents the coordinate of the domain of $\gamma$. Then $$\p^o_o[\cP{p}] = \sum_{\substack{q \in \crit^o(f)\\ \ind(\cP{p}) - \ind(\cP{q}) = 1}} \sum_{\gamma \in \M{M}{p}{q}}\orientation{\gamma}[q].$$

    Note that for any $h \in G$ maps the moduli space $\M{M}{p}{q}$ to $\mathcal{M}(M;h\cP{p},h\cP{q})$ and changes the orientation by: $$\orientation{h\gamma} = \sigma(h, \cP{p})\sigma(h,\cP{q})\orientation{\gamma}.$$
    Then a direct calculation shows $dh[\cP{p}] = hd[\cP{p}]$.

\end{proof}

\begin{example}[Upright torus]\label{example: torus}
\begin{figure}[h]
     \centering
    \begin{subfigure}[b]{.4\textwidth}
        \centering
         \begin{tikzpicture}[scale = 0.8]
    \draw[blue] (0,0) ellipse (2cm and 3cm);
    \draw[blue] (-0.1, 1.2) .. controls (0.5,1) and (0.5,-1) .. (-0.1, -1.2);
    \draw[blue] (0, 1.1) .. controls (-0.5, 1) and (-0.5, -1) .. (0,-1.1);

    \draw[orange, fill] (0, 3) circle (0.05);
    \node[anchor = south] at (0, 3) {$\cP{a}$};
    \draw[fill] (0, 1.05) circle (0.05);
    \node[anchor = south] at (0, 1.1) {$\cP{b}$};

    \draw[orange, fill] (0, -1.05) circle (0.05);
    \node[anchor = north] at (0, -1.1) {$\cP{c}$};

    \draw[fill] (0, -3.) circle (0.05);
    \node[anchor = north] at (0, -3.1) {$\cP{d}$};

    \draw[] (0.2, 2.98) .. controls (-0.7, 3.1) and (-0.6, 1).. (0, 1.05);
    \node[red] at (-0.75, 2) {\footnotesize{$u_{+}$}};

    \draw[dotted] (0.2, 2.98) .. controls (0.75, 2.3) and (0.55, 0.9).. (0, 1.05);
    \node[red] at (0.85, 1.95) {\footnotesize $u_{-}$};

    \node[red] at (-0.6, 0) {\footnotesize $v$};

    \node[red] at (0.6, 0) {\footnotesize $w$};
    
    \end{tikzpicture}
        \caption{Upright torus}
        \label{fig:upright torus in plane}
    \end{subfigure}
 \begin{subfigure}[b]{.4\textwidth}
        \centering
             \begin{tikzpicture}[
         > = Stealth,
->-/.style = {decoration={markings,
                          mark=at position 0.55 with {\arrow{>}}
                         },
               postaction={decorate}
              },
->>-/.style = {decoration={markings,
                         mark=at position 0.55 with {\arrow{>>}}
                        },
               postaction={decorate}
              },
dot/.style = {circle, draw, fill, inner sep=0.8pt,
              label=#1}
                        ]
\draw[->-,gray] (-0.5,2) -- ++ (1,0);
\draw[->-,gray] (-.5,-2) -- ++ (1,0);
\draw[->>-,gray] (-2.5,-0.5) -- ++ (0,1);
\draw[->>-,gray] (2.5, -0.5) -- ++ (0,1);

 \node[blue, anchor = north] at (2.25,2.25) {$M^G$};
\node (n11) [dot={left:$\cP{b}$}] at (-1.5, 1.5) {};
\node (n12) [orange, dot={left: $\cP{a}$}] at (-1.5, 0.0) {};
\node (n13) [dot={left:$\cP{b}$}] at (-1.5,-1.5) {};
\node (n21) [orange, dot={below left:$\cP{c}$}] at (0, 1.5) {};
\node (n22) [dot={below left:$\cP{d}$}] at (0, 0.0) {};
\node (n23) [orange, dot={below left:$\cP{c}$}] at (0,-1.5) {};
\node (n31) [dot={right:$\cP{b}$}] at (1.5, 1.5) {};
\node (n32) [orange, dot={right:$\cP{a}$}] at (1.5, 0.0) {};
\node (n33) [dot={right:$\cP{b}$}] at (1.5,-1.5) {};
\foreach \y in {1,2,3}
{
    \draw[->-,blue] (n1\y) -- (n2\y);
    \draw[->-,blue] (n3\y) -- (n2\y);
}
\draw[->-] (n12) -- (n11);
\draw[->-] (n12) -- (n13);
\draw[->-] (n21) -- (n22);
\draw[->-] (n23) -- (n22);
\draw[->-] (n32) -- (n31);
\draw[->-] (n32) -- (n33);

        \node[anchor = north] at (-1.3, 0.45) {\footnotesize $+$};
        \node[anchor = north] at (-1.3, 0.05) {\footnotesize $-$};
        \node[anchor = north] at (1.3, 0.45) {\footnotesize $+$};
        \node[anchor = north] at (1.3, 0.05) {\footnotesize $-$};
         \node[anchor = north] at (-1.3, 1.95) {\footnotesize $-$};
         \node[anchor = north] at (-1.3, 1.55) {\footnotesize $+$};
         \node[anchor = north] at (1.3, 1.95) {\footnotesize $-$};
         \node[anchor = north] at (1.3, 1.55){\footnotesize $+$};
         \node[anchor = north] at (-1.3, -1.05) {\footnotesize $-$};
         \node[anchor = north] at (-1.3, -1.45) {\footnotesize $+$};
         \node[anchor = north] at (1.3, -1.05) {\footnotesize $-$};
         \node[anchor = north] at (1.3, -1.45) {\footnotesize $+$};
        \node[anchor = north] at (0.25, 1.95) {\footnotesize $-$};
        \node[anchor = north] at (0.25, 1.55) {\footnotesize $+$};
         \node[anchor = north] at (0.25, -1.05) {\footnotesize $-$};
        \node[anchor = north] at (0.25, -1.45) {\footnotesize $+$};

        \node[anchor = north] at (0.25, 0.45) {\footnotesize $+$};
        \node[anchor = north] at (0.25, 0.05) {\footnotesize $-$};

        \node[red] at (-1.75, .75) {\footnotesize $u_{+}$};
        \node[red] at (-1.75, -0.75) {\footnotesize $u_{-}$};
        \node[red] at (-.75, 1.75) {\footnotesize $v$};
        \node[red] at (.8, 1.75) {\footnotesize $w$};
    \end{tikzpicture}
        \caption{Polygon representation of torus}
        \label{fig:dynamics on torus}
    \end{subfigure}
    \caption{Here we depict a choice of orientations for the upright torus with trajectories $u_{\pm}$ flowing from the $\pm$ side of $M^G$. In addition, $v$ and $w$ are the two orientation--preserving trajectories from $\cP{b}$ to $\cP{c}$.
    }
    \label{torus example}
\end{figure}

Let $M$ be an upright torus lying in the plane, with four critical points labeled $\cP{a},\cP{b},\cP{c},\cP{d}$ from top to bottom. The diffeomorphism $\phi$ reflects the front and back of $M$, and the fixed point set $M^G$ is a disjoint union of two circles. The critical points $\cP{a}$ and $\cP{c}$ are unstable, while the critical points $\cP{b}$ and $\cP{d}$ are stable. 

At each critical point $\cP{p} \in \{\cP{a}, \cP{b}, \cP{c}, \cP{d}\}$, we choose the unit vector $\mathfrak o_p \in T_{\cP{p}}M$ perpendicular to $T_{\cP{p}}M^G$ to point to the front side of $M$.

Now we compute $(\boldsymbol{C}, \boldsymbol{\p})$, where 
\[
\boldsymbol{C} = \Z\langle \varnothing \rangle \oplus \Z\langle \cP{b},\cP{d} \rangle \oplus \Z\langle \cP{a},\cP{c} \rangle[-1] \oplus \Z \langle \cP{a}_{+}, \cP{c}_{+} \rangle \oplus \Z \langle \cP{a}_{-}, \cP{c}_{-} \rangle.
\]
Starting with an index 2 critical point, we have $\boldsymbol\p \cP{a}_{+} = \p_{s}^{u+} \cP{a}_{+} - \Pi_{+} \cP{a}_{+} = \cP{b} - \cP{a}$, where all other terms vanish. Here, $\p_s^{u+}$ counts the number of flowlines from $\cP{a}$ to $\cP{b}$ from the positive side of $\cP{a}$ (in agreement with $\mathfrak o_{\cP{a}}$). 

To illustrate how the remaining terms vanish, we examine the count of broken flowlines from $\cP{a}_{+}$ to $\cP{c}_{+}$. This count is zero because the broken trajectories $(u_{+}, v)$ and $(u_{+}, w)$ cancel each other, as depicted in Figure~\ref{torus example}.

Similarly, we have the remaining terms:
\begin{align*}
      \boldsymbol \p \cP{a}_{+}&=\cP{b}-\cP{a} & \boldsymbol{\p} \cP{a}&=0 & \boldsymbol \p \cP{c}_{+}&=\cP{d}-\cP{c} &
\boldsymbol \p \cP{c}&=0\\
 \boldsymbol \p \cP{a}_-&=\cbu{\cP{a}-\cP{b}} \cb & \boldsymbol \p \cP{b}&=0 &   \boldsymbol \p \cP{c}_{-}&=-\cP{d}+\cP{c}& \boldsymbol \p \cP{d}&=0.\\ 
\end{align*}
This gives us the homology of the torus \[{\cbu{H_0(T^2; \Z)= \Z, \quad \quad H_1(T^2; \Z)= \Z^2, \quad \quad H_2(T^2; \Z)= \Z}}\] as expected.
\end{example}

\begin{remark}
    Note that $\m{a}{c} = \emptyset$.  This means that $\m{a}{b}$ and $\m{b}{c}$ do not glue. But some of these broken trajectories may glue after a Kuranishi perturbation. This is the approach of \cite{bao2024computable}, which is not what we do in this paper.
\end{remark}

\begin{example}[Klein bottle]\label{example: klein bottle}
    \begin{figure}[h]
        \centering
        \begin{tikzpicture}[
             > = Stealth,
    ->-/.style = {decoration={markings,
                              mark=at position 0.55 with {\arrow{>}}
                             },
                   postaction={decorate}
                  },
    ->>-/.style = {decoration={markings,
                             mark=at position 0.55 with {\arrow{>>}}
                            },
                   postaction={decorate}
                  },
    dot/.style = {circle, draw, fill, inner sep=0.8pt,
                  label=#1}
                            ]
    \draw[->-,gray] (-0.5,2) -- ++ (1,0);
    \draw[->-,gray] (-.5,-2) -- ++ (1,0);
    \draw[->>-,gray] (-2.5,0.5) -- ++ (0,-1);
    \draw[->>-,gray] (2.5, -0.5) -- ++ (0,1);
    
    \node[blue, anchor = north] at (2.25,2.25) {$M^G$};
    \node (n11) [dot={left:$\cP{d}$}] at (-1.5, 1.5) {};
    \node (n12) [orange, dot={left:$\cP{c}$}] at (-1.5, 0.0) {};
    \node (n13) [dot={left:$\cP{d}$}] at (-1.5,-1.5) {};
    \node (n21) [orange, dot={below left:$\cP{a}$}] at (0, 1.5) {};
    \node (n22) [dot={below left:$\cP{b}$}] at (0, 0.0) {};
    \node (n23) [orange, dot={below left:$\cP{a}$}] at (0,-1.5) {};
    \node (n31) [dot={right:$\cP{d}$}] at (1.5, 1.5) {};
    \node (n32) [orange, dot={right:$\cP{c}$}] at (1.5, 0.0) {};
    \node (n33) [dot={right:$\cP{d}$}] at (1.5,-1.5) {};
    
    \foreach \y in {1,2,3} {
        \draw[->-,blue] (n2\y) -- (n1\y);
        \draw[->-,blue] (n2\y) -- (n3\y);
    }
    \draw[->-] (n12) -- (n11);
    \draw[->-] (n12) -- (n13);
    \draw[->-] (n21) -- (n22);
    \draw[->-] (n23) -- (n22);
    \draw[->-] (n32) -- (n31);
    \draw[->-] (n32) -- (n33);
    
    \node[anchor = north] at (-1.3, 0.45) {\footnotesize $-$};
    \node[anchor = north] at (-1.3, 0.05) {\footnotesize $+$};
    \node[anchor = north] at (1.3, 0.45) {\footnotesize $+$};
    \node[anchor = north] at (1.3, 0.05) {\footnotesize $-$};
    \node[anchor = north] at (-1.3, 1.95) {\footnotesize $-$};
    \node[anchor = north] at (-1.3, 1.55) {\footnotesize $+$};
    \node[anchor = north] at (1.3, 1.95) {\footnotesize $-$};
    \node[anchor = north] at (1.3, 1.55){\footnotesize $+$};
    \node[anchor = north] at (-1.3, -1.05) {\footnotesize $-$};
    \node[anchor = north] at (-1.3, -1.45) {\footnotesize $+$};
    \node[anchor = north] at (1.3, -1.05) {\footnotesize $-$};
    \node[anchor = north] at (1.3, -1.45) {\footnotesize $+$};
    \node[anchor = north] at (0.25, 1.95) {\footnotesize $-$};
    \node[anchor = north] at (0.25, 1.55) {\footnotesize $+$};
    \node[anchor = north] at (0.25, -1.05) {\footnotesize $-$};
    \node[anchor = north] at (0.25, -1.45) {\footnotesize $+$};
    \node[anchor = north] at (0.25, 0.45) {\footnotesize $+$};
    \node[anchor = north] at (0.25, 0.05) {\footnotesize $-$};
    
    \node[red] at (-0.2, 0.7) {\footnotesize $u_{+}$};
    \node[red] at (-0.2, -0.7) {\footnotesize $u_{-}$};
    \node[red] at (-.75, .25) {\footnotesize $v$};
    \node[red] at (.75, 0.25) {\footnotesize $w$};
    
        \end{tikzpicture}
        \caption{A polygon representation of the Klein bottle, formed by identifying the gray arrows.}
    \label{fig:dynamics on Klein bottle}
    \end{figure}
    
    Consider a Klein bottle with the same setup as in Example~\ref{example: torus}. The group $G=\Z_2$ acts such that the fixed point set $M^G$ consists of two distinct loops, as shown in Figure~\ref{fig:dynamics on Klein bottle}. The critical points $\cP{a}$ and $\cP{c}$ are unstable, while the critical points $\cP{b}$ and $\cP{d}$ are stable.
    
    We choose $\mathfrak o_{\cP{p}}$ for each critical point $\cP{p} \in \crit(f)$ as indicated by the $+$ sign in the figure. Using the same vector space
    \[
    \boldsymbol{C} = \Z\langle \varnothing \rangle \oplus \Z\langle \cP{b}, \cP{d} \rangle \oplus \Z\langle \cP{a}, \cP{c} \rangle[-1] \oplus \Z \langle \cP{a}_{+}, \cP{c}_{+} \rangle \oplus \Z \langle \cP{a}_{-}, \cP{c}_{-} \rangle,
    \]
    we compute the homology. For example, we have $\boldsymbol\p \cP{a}_{+} = \cP{b} - \cP{a} + \cP{c}_++\cP{c}_-$. The term $\cP{b} - \cP{a}$ is straightforward to obtain, and $\cP{c}_\pm$ results from counting the broken trajectories $(u_{+}, w)$ and $(u_{+}, v)$ from $\cP{a}_{+}$ to $\cP{c}_{\pm}$ respectively. 
    
    Computing this for all the critical points, we obtain:
   
\begin{align*}
    \boldsymbol{\p} \cP{a}_{+} &= \cP{b} - \cP{a} + {\cbu{\cP{c}_+ + \cP{c}_-} }&  \boldsymbol{\p} \cP{a} &= 0 &  \boldsymbol{\p} \cP{c}_{+} &= \cP{d} - \cP{c} &  \boldsymbol{\p} \cP{c} &= 0 \\
     \boldsymbol{\p} \cP{a}_{-} &= -\cP{b} + \cP{a} + {\cbu{\cP{c}_+ + \cP{c}_-}}  & \boldsymbol{\p} \cP{b} &= 0 &  \boldsymbol{\p} \cP{c}_{-} &= -\cP{d} + \cP{c} & \boldsymbol{\p} \cP{d} &= 0.\\
\end{align*}
This yields the homology of the Klein bottle: \[{\cbu{H_0(K; \Z)= \Z, \quad \quad H_1(K; \Z)= \Z\oplus \Z_2, \quad \quad H_2(K; \Z)= 0}}.\]
\end{example}

\cbu  
We conjecture that the homology of $(\boldsymbol{C}, \boldsymbol{\p})$ is isomorphic to the homology of $M$. In fact, we have a stronger conjecture.
Let $(f,g)$ be a stably Morse-Smale pair.
Suppose that the answer to Question~\ref{question: name of stably ms} is positive. 
Then the Morse-Smale pair $(f^\epsilon, g)$ gives the usual Thom-Smale-Witten complex $(C^\epsilon, \p^\epsilon)$. 
\begin{conjecture}
    The chain complexes $(C^\epsilon, \p^\epsilon)$ and $(\boldsymbol{C}, \boldsymbol{\p})$ are $\Z[G]$-isomorphic.
\end{conjecture}

\cb 

\subsection{Borel equivariant Morse homology}\label{section: equivariant morse homology}
In this section, we recall the definition of Borel's construction of equivariant homology. For other types of equivariant homology, see \cite{bao2024morse}.
Let $(C, d)$ be a chain complex over the ring $R$ with a finite group $G$-action, which means $G$ preserves the grading of $C$ and commutes with $d$.

Let $(P, d^P)$ be a projective resolution of $R$ over $R[G]$. We denote
$$E_{i,j} \coloneqq P_{i} \otimes_{R[G]} C_j,$$
for $i, j \in \N$, where $P_{i}$ and $C_j$ are regarded as $R[G]$-modules. Let $d^<_{i,j}: E_{i,j} \to E_{i-1,j}$ be the map induced by $d^{P}_{i}: P_{i} \to P_{i-1}$, and $d^\vee _{i,j}: E_{i,j} \to E_{i,j-1}$ be the map induced by $d_j: C_j \to C_{j-1}$ multiplied by the factor $(-1)^i$. Then, $d^<_{i,j}$ and $d^\vee_{i,j}$ commute with the multiplication by elements in $R[G]$ and form a double complex.

We define the chain complex $(C^G, d^G)$ by $C^G = \oplus_k C^G_k$, where 
$$C_k^G  = \oplus_{i+j = k} E_{i,j}\quad \text{ and } \quad d^G|_{E_{i,j}}  = d^<_{i,j} + d^\vee_{i,j}.$$
The corresponding $G$-equivariant homology group is:
$$H^G= \ker d^G / \op{Im} d^G.$$

\cbu 
\begin{conjecture}
When $(C,d) = (\boldsymbol{C}, \boldsymbol{\p})$ as in Section~\ref{section: equivariant chain complex}, $H^G$ computes the Borel homology $H_{\text{Borel}}^G(M; \Z) = H(M \times_G EG)$.
\end{conjecture}
\cb

In the case when $R = \Z$, and $G = \langle \phi \rangle$, then $R[G] = \Z \langle \mathbbm 1, \phi \rangle$. We can take the projective resolution as follows:
\[
\dots \xrightarrow{d^P_3} \Z[G] \xrightarrow{d^P_2} \Z[G] \xrightarrow{d^P_1} \Z[G] \xrightarrow{d^P_0} \Z \to 0.
\]
The $\Z[G]$-linear differentials are defined by
$d^P_0(\mathbbm{1}) = 1$ and $d^P_i(\mathbbm{1}) = \mathbbm{1} + (-1)^{i+1} \phi$, for $i \geq 1$. 
Then the double complex $E_{i,j}$ becomes:
\begin{equation*}
    \begin{tikzcd}
            & \vdots \arrow[d, "d_2"] & \vdots \arrow[d, "\cbu - d_2"] \\
        0  & C_1 \arrow[l] \arrow[d, "d_1"] & C_1 \arrow[l, "\mathbbm 1 - \phi"] \arrow[d, "\cbu -d_1"] & \arrow[l, "\mathbbm 1 + \phi"] \dots \\
        0  & C_0 \arrow[l]\arrow[d] & C_0 \arrow[l, "\mathbbm 1 - \phi"] \arrow[d] & \arrow[l, "\mathbbm 1 + \phi"] \dots . \\
            & 0 & 0
    \end{tikzcd}
\end{equation*}

\cbu We now compute the equivariant homology of Examples~\ref{example: torus} and~\ref{example: klein bottle}. 

\begin{example}[Upright torus]Take the chain complex $(\boldsymbol{C}, \boldsymbol{\p})$ as in Example~\ref{example: torus}. Then we have that \[H_0^G(T^2; \Z)=\Z, \quad\quad H_1^G(T^2; \Z)=\Z \oplus \Z_2^2, \quad\quad H_k^{G}(T^2; \Z)=\Z_2^2, \text{ for all } k\geq 2.\]
\end{example}

\begin{example}[Klein bottle]Consider the chain complex $(\boldsymbol{C}, \boldsymbol{\p})$ as in Example~\ref{example: klein bottle}. Then we have that \[H_0^G(K; \Z)=\Z, \quad\quad H_1^G(K; \Z)=\Z \oplus \Z_2^2, \quad\quad H_k^{G}(K; \Z)=\Z_2^2, \text{ for all } k\geq 2.\]
\end{example}

\section{Proof of d square is zero}\label{section: d square is zero}
In this section, we prove Proposition~\ref{prop: fatD square is zero} and Theorem~\ref{thm: bold d square is zero}. We focus on Theorem~\ref{thm: bold d square is zero}, as Proposition~\ref{prop: fatD square is zero} follows as a consequence. \cbu One would ideally like to construct an isomorphism \cb between $(\boldsymbol{C}, \boldsymbol{\p})$ and the Thom-Smale-Witten complex of the pair $(f^\epsilon, g)$, \cbu but this is difficult for the following reason: While the map on generators is straightforward, it is challenging to show that this map commutes with the boundary maps, which requires the \cb obstruction bundle gluing from broken trajectories of the pair $(f,g)$ to smooth trajectories of the perturbed pair $(f^\epsilon, g)$. In this note, we provide a more brute-force proof using the standard obstruction bundle gluing method without reference to $f^\epsilon$.

\subsection{The statement of obstruction bundle gluing} 
We first review the Fredholm setup. For any critical points $\cP{p}, \cP{q} \in \crit(f)$, define the Banach manifold 
\[
    \widetilde{\mathcal B}(\cP{p}, \cP{q}) = \{\gamma \in L_1^2(\R, M) ~|~ \lim_{\vt \to -\infty} \gamma(\vt) = \cP{p}, \lim_{\vt\to \infty} \gamma(\vt) = \cP{q} \},
\]
where $L_1^2(\R, M)$ is defined by embedding $M$ into $\R^k$ for some large $k$.
We define the Banach bundle $\widetilde{\mathcal E} \to \widetilde{\mathcal B}$ by 
$\widetilde{\mathcal E}_\gamma = L^2(\gamma^* TM)$. Denote by $\mathcal L: \widetilde{\mathcal B} \to \widetilde{\mathcal E}$ the section defined by 
\[
\mathcal L \gamma = \frac{d}{d\vt}\gamma + \op{grad}_{f,g}(\gamma).
\]
\begin{lemma}[Lemma 2.10 \cite{bao2024computable}]
    The section $\mathcal L$ intersects the zero section cleanly if and only if the descending manifold of $\cP{p}$ intersects the ascending manifold of $\cP{q}$ cleanly.
\end{lemma}

For any $\gamma \in \lm = \mathcal L^{-1}(0)$, denote the linearization of $\mathcal L$ at $\gamma$ by $D_\gamma$.
It is well-known that $D_\gamma: L_1^2(\gamma^* TM) \to L^2(\gamma^* TM)$ is a Fredholm operator of index $\ind(\cP{p}) - \ind(\cP{q})$ given by
\[
D_\gamma \xi = \nabla_\vt \xi + \nabla_\xi \op{grad}_{f,g}(\gamma),
\]for $\xi \in L^2_1(\gamma^{*}TM).$
Let $D_\gamma^*: L_{-1}^2(\gamma^* TM) \to L^2(\gamma^* TM)$ be the $L^2$-adjoint operator of $D_\gamma$ given by 
\[
D_\gamma^* \eta = -\nabla_\vt \eta + \nabla_\eta \op{grad}_{f,g}(\gamma),
\]for $\eta \in L^2_{-1}(\gamma^{*}TM).$
We write $\widetilde{\obstructionBundle} \to \widetilde{\lm}$ to be the obstruction bundle, where $\widetilde{\obstructionBundle}_\gamma = \ker D_\gamma^*$. It is easy to see that 
\[
    \widetilde \obstructionBundle_\gamma  \oplus \op{im}D_\gamma = \widetilde{\mathcal E}_\gamma.
\]

The group $\R$ acts on $\widetilde{\mathcal E} \to \widetilde{\mathcal B}$ by translating the domain. We denote the quotient bundle as $\mathcal E \to \mathcal B$. Since $\mathcal L$ is invariant under the $\R$-action, $\mathcal L$ descends to a section of $\mathcal E \to \mathcal B$, and is still denoted by $\mathcal L$.
We write $\obstructionBundle$ to be the bundle obtained by quotienting $\widetilde{\obstructionBundle}$ by $\R$.

It is well-known that $\lm$ can be compactified by adding broken trajectories. We now study which broken trajectories appear in $\p \lm = \overline{\lm} \backslash \lm$ in the case when $\ind(\cP{p}) - \ind(\cP{q}) = 2$.

\begin{definition}[Broken trajectory]
    A \emph{broken trajectory} from $\cP{p}$ to $\cP{q}$ consists of $$([u_1], \dots, [u_m]) \in \lm(\cP{p}_{0}, \cP{p}_1) \times \dots \times \lm(\cP{p}_{m-1}, \cP{p}_m),$$
    for some $m > 1$,
    where $\cP{p}_0, \dots, \cP{p}_m$ are critical points in $M$ with $\cP{p}_0 = \cP{p}$ and $\cP{p}_m = \cP{q}$.
\end{definition}

\begin{definition}[Simple]\label{def: simple}
    We say a broken trajectory $([u_1], \dots, [u_m])$ is \emph{simple} if it satisfies:
    \begin{enumerate}[label=(A\arabic*)]
        \item $\ind(\cP{p}_{i-1}) - \ind(\cP{p}_i) \in \{0,1\}$.
        \item $[u_i]$ is isolated in $\lm(\cP{p}_{i-1}, \cP{p}_i)$ and $[u_i]$ is cleanly cut out.
        \item \label{condition: special assumption} For any $i \in \{1, \dots, m\}$ such that $\ind(\cP{p}_{i-1}) - \ind(\cP{p}_i) = 0$, there exists a codimension one submanifold $F_i \subset M$ such that:
        \begin{enumerate}[label=(\roman*)]
            \item $F_i$ is tangent to $\op{grad}_{f,g}$.
            \item $u_i$ is contained in $F_i$.
            \item For any $u_j \in \lm(\cP{p}_{j-1}, \cP{p}_j)$ with $j = i-1, i+1$, $u_j$ is not contained in $F_i$.
        \end{enumerate}
    \end{enumerate}
\end{definition}

Note that Condition~\ref{condition: special assumption} implies that we can define the asymptotic behavior of $u_{i-1}$ and $u_{i+1}$ in the normal direction of $F_i$. In particular, we define
\[
u_{i-1}^\infty = \lim_{\vt\to \infty} {n_{i-1}(\vt)}/{|n_{i-1}(\vt)|} \in T_{\cP{p}_{i-1}}M
\]
and
\[
u_{i+1}^{-\infty} = \lim_{\vt\to -\infty} {n_{i+1}(\vt)}/{|n_{i+1}(\vt)|} \in T_{\cP{p}_i}M
\]
as before, where $n_{i-1}$ and $n_i$ are the components of $u_{i-1}$ and $u_i$ in the normal direction of $F_i$.

In the proof of Theorem~\ref{thm: bold d square is zero}, no consecutive moduli spaces are obstructed. However, in Definition~\ref{def: simple} we include the slightly general setup for the application in Section~\ref{section: beyond reflection}. 

\begin{definition}[Gluable] \label{def: gluable}
    We say a simple broken trajectory $([u_1], \dots, [u_m])$ is \emph{gluable} if for any $i$ such that $\ind(\cP{p}_{i-1}) - \ind (\cP{p}_i) = 0$, we have $1 < i < m$ and 
    \[
    \langle u_{i-1}^\infty, \xi^{-\infty} \rangle \cdot \langle \xi^\infty, u_{i+1}^{-\infty} \rangle > 0,
    \]
    where $\xi$ is a nonzero element in $\ker D_{u_i}^*$, i.e., it satisfies Equation~\eqref{eqn: kernel of D*}.
\end{definition}

To orient the moduli spaces, we fix an orientation $\orientation{\mathscr D_{\cP{p}}}$ for the descending manifold $\mathscr D_{\cP{p}}$ of each critical point $\cP{p}$. For each $i$ as in Condition~\ref{condition: special assumption}, we choose a unit vector $\mathfrak o_{\cP{p}_i}$ orthogonal to $F_i$, and we also choose orientations $\orientation{\mathscr D_{\cP{p}_i} \cap F_i}$ for the descending manifolds of $\cP{p}_i$ such that 
\[
\orientation{\mathscr D_{\cP{p}_i}} = \mathfrak o_{\cP{p}_i} \oplus \orientation{\mathscr D_{\cP{p}_i} \cap F_i}.
\]
Then, for each gluable broken trajectory $([u_1], \dots, [u_m])$, we can associate an number in $\{1, -1\}$ by $\orientation{u_1} \cdots \orientation{u_m}$. Each $\orientation{u_i}$ is assigned using the orientations of the descending manifolds in $M$ if $\ind(u_i) = \ind(\cP{p}_{i-1}) - \ind(\cP{p}_i) = 1$, and the orientations of the descending manifolds in $F_i$ if $\ind(u_i) = \ind(\cP{p}_{i-1}) - \ind(\cP{p}_i) = 0$.

\begin{remark}
    In the case where we have a manifold $N$ with boundary $\p N$, we glue $N$ with itself to obtain $M \coloneqq  N \cup \overline{N}$, and we choose orientations $\mathfrak o_\cP{p} \in T_\cP{p} \p N$ for each critical point $\cP{p} \in \mathcal N$ such that they all point outwards. Here, a simple broken trajectory consists of at most three pieces, i.e., $m = 2, 3$, by an index calculation, and all the simple broken trajectories in $N$ are gluable. The following theorem is a generalization of Lemma 2.4.3 in \cite{kronheimer2007Monopoles}.
\end{remark}

\begin{theorem}[Gluable trajectories]\label{thm: gluing end goal}
    Assume the following conditions:
    \be
    \item Critical points $\cP{p}$ and $\cP{q}$ are two critical points of $f$ such that $\ind \cP{p} - \ind \cP{q} = 2$.
    \item The moduli space $\lm (\cP{p}, \cP{q})$ is transversely cut out.
    \item All the broken trajectories from $\cP{p}$ to $\cP{q}$ are simple.
    \ee 
    Then the moduli space $\lm(\cP{p}, \cP{q})$ is an oriented one-dimensional manifold. 
    \cbu The boundary of its compactification is cobordant to the space of \cb all the gluable simple broken trajectories from $\cP{p}$ to $\cP{q}$ with the sign correction
    $$(-1)^m \prod_{i: \ind(\cP{p}_{i-1}) = \ind(\cP{p}_i)}   \langle u^{-\infty}_{i+1},\mathfrak o_{\cP{p}_{i}} \rangle,$$
    for each broken trajectory $([u_1], \dots, [u_m])$.
\end{theorem}

The proof of Theorem~\ref{thm: gluing end goal} is postponed to the next subsection. Assuming this, we prove Theorem~\ref{thm: bold d square is zero}.
\begin{proof}[Sketch of proof of Theorem~\ref{thm: bold d square is zero}]
    As a matrix, $\boldsymbol{\p}^2$ is a $(5 \times 5)$ matrix. Some of the terms are automatically zero. Checking the remaining terms being zero is quite similar: they all count gluable broken trajectories, and we demonstrate this in the following case:

    \emph{Case $(\boldsymbol{\p}^2)_{2,1} = 0$}:
    For any $\cP{p} \in \crit^o(f)$ and $\cP{s} \in \crit^s(f)$, with $\ind (\cP{p}) - \ind (\cP{s}) = 2$, we consider the moduli space $\lm = \m{p}{s}$. Then $\lm$ is a smooth, oriented, one-dimensional manifold. $\lm$ can be compactified by adding broken trajectories of the following types:
    \be 
        \item $\m{p}{w} \times \m{w}{s}$, for all $\cP{w} \in \crit^o(f)$ such that $\ind(\cP{w}) = \ind(\cP{p}) - 1$;
        \item $\m{p}{w} \times \m{w}{s}$, for all $\cP{w} \in \crit^s(f)$ such that $\ind(\cP{w}) = \ind(\cP{p}) - 1$;
        \item $\m{p}{q} \times \m{q}{r} \times \m{r}{s}$, for all $\cP{q} \in \crit^s(f)$ such that $\ind(\cP{p}) - \ind(\cP{q}) = 1$, and all $\cP{r} \in \crit^u(f)$ such that $\ind(\cP{r}) = \ind(\cP{p}) - 1$.
    \ee 
    Note that the moduli space $\m{q}{r}$ in (3) consist of trajectories only inside $M^G$. Moreover, it is easy to check that no other types of broken trajectories can serve as the boundary of $\lm$. In other words, all the broken trajectories from $\cP{p}$ to $\cP{q}$ are simple. 
    
    Then, by Theorem~\ref{thm: gluing end goal}, $\lm$ can be compactified by adding gluable broken trajectories, i.e., those of type (1), type (2), and type (3) with the additional condition from Definition~\ref{def: gluable}.
    
    We then claim:
    \[
    \langle \boldsymbol{\p}^2 \cP{p}, \cP{q} \rangle = \langle (\p^o_s \p^o_o + \bar\p^s_s \p^o_s + \p^{u+}_s (P^s_u \p^o_{s+} + R^s_u \p^o_{s-}) - \p^{u-}_s (P^s_u \p^o_{s-} + R^s_u \p^o_{s+})) \cP{p}, \cP{q} \rangle 
    \] 
    counts all the gluable broken trajectories from $\cP{p}$ to $\cP{q}$. In particular, the first two terms correspond to broken trajectories of type (1) and (2), while the remaining terms correspond to broken trajectories of type (3). The minus sign in the last term is due to the sign correction. 
    
    Therefore, $\langle \boldsymbol{\p}^2 \cP{p}, \cP{q} \rangle$ counts $\p \lm$, which is $0$ as the boundary of one-dimensional manifolds.

    \end{proof}

\subsection{Details of obstruction bundle gluing}

Given a simple broken trajectory $$([u_1], \dots, [u_m])$$ and gluing parameters $T_1, \dots , T_{m-1} \gg 0$, 
we construct a pre-gluing map as follows.

Choose a sufficiently small $\varepsilon > 0$, and let $i \in \{1, \dots, m\}$. Around each critical point $\cP{p}_i$, we choose a ball of radius $\varepsilon$. We choose representatives $(u_1, \dots, u_m)$ of $([u_1], \dots, [u_m])$ such that $f(u_i(0)) = \frac{1}{2}(f(\cP{p}_{i-1}) + f(\cP{p}_{i+1}))$. We choose constants $a_i < b_i$ such that $u_i$ leaves the $\varepsilon$-ball centered at $\cP{p}_{i-1}$ at time $a_i$ and enters the $\epsilon$-ball centered at $\cP{p}_{i}$ at time $b_i$. 
We construct the domain of the glued curve as the disjoint union of the following intervals with adjacent endpoints identified:
$$(-\infty, b_1],[-2T_1, 2T_1], [a_2, b_2], \dots, [a_{m-1}, b_{m-1}],[-2T_{m-1}, 2T_{m-1}],[a_m, \infty).$$
In particular, we identify the union (with consecutive end points identified) of $[-2T_{i-1}, 2T_{i-1}], [a_i, b_i], [-2T_i, 2T_i]$ with \cbu $[a_i-4T_{i-1}, b_i + 4T_{i}]$ \cb when viewed as part of the domain of $u_i$.
Let $0 < r < \min_i T_i$ be a fixed constant.
We choose a bump function \cbu $\beta_i: [a_i-4T_{i-1}, b_i + 4T_{i}] \to [0,1]$ \cb such that
\be 
    \item $\beta_i(\vt) = 0$ for \cbu $\vt \leq a_i-3T_{i-1} - r$ or $\vt > b_i + 3T_{i} + r$; \cb
    \item $\beta_i(\vt) = 1$ for \cbu $a_i-3T_{i-1} + r < \vt < b_i + 3T_{i} - r$; \cb
    \item $|\beta'_i(\vt)| < 1/r$.
\ee  
See Figure~\ref{fig: bump function}.
Here, we use the convention that $a_1 = -\infty$ and $b_m = \infty$. 

Before we define a pregluing, we introduce an auxiliary Riemannian metric on $M$ to identify a tangent vector with a point in $M$ using the exponential map. The metric is chosen to be the Euclidean metric in the $2\varepsilon$-balls around each critical points. If $u$ is a trajectory and $\xi$ is a section of $u^* TM$, then we write $u + \xi$ to mean the image of exponential map of $\xi$ at $u$ with respect to the auxiliary metric.

\begin{figure}
\centering
    \begin{tikzpicture}[scale = 0.7]
        \draw[->] (-7, 0) -- (7, 0);
        \draw[] (-6, -1) -- (-6, 0.25);
        \draw[] (6, -1) -- (6, 0.25);
        \draw[] (-2, -1) -- (-2, 0.25);
        \draw[] (2, -1) -- (2, 0.25);

        \draw[->] (-1, -0.6) -- (-2, -0.6);
        \draw[->] (1, -0.6) -- (2, -0.6);
        \node[] at (0, -0.6) {$b_i - a_i$};

        \draw[->] (-5, -0.5) -- (-6, -0.5);
        \draw[->] (-3, -0.5) -- (-2, -0.5);
        \node[] at (-4, -0.5) {$4T_{i-1}$};

        \draw[->] (5, -0.5) -- (6, -0.5);
        \draw[->] (3, -0.5) -- (2, -0.5);
        \node[] at (4, -0.5) {$4T_{i}$};

        \draw[red] (-5.5, 0.5) .. controls (-5, 0.5) and (-5, 1.5) .. (-4.5, 1.5);
        \draw[red] (5.5, 0.5) .. controls (5, 0.5) and (5, 1.5) .. (4.5, 1.5);
        \draw[red] (-4.5, 1.5) -- (4.5, 1.5);
        \draw[red] (-7, 0.5)-- (-5.5, 0.5);
        \draw[red] (7, 0.5)-- (5.5, 0.5);
        \node[red] at (0, 1.8) {$\beta_i$};

        \draw[blue] (2.5, 0.4) .. controls (3, 0.4) and (3, 1.4) .. (3.5, 1.4);
        \draw[blue] (2, 0.4) -- (2.5, 0.4);
        \draw[blue] (3.5, 1.4) -- (7,1.4);
        \node[blue] at (6, 1.7) {$\beta_{i+1}$};

        \draw[cyan] (-2.5, 0.4) .. controls (-3, 0.4) and (-3, 1.4) .. (-3.5, 1.4);
        \draw[cyan] (-2, 0.4) -- (-2.5, 0.4);
        \draw[cyan] (-3.5, 1.4) -- (-7,1.4);
        \node[cyan] at (-6, 1.7) {$\beta_{i-1}$};
    \end{tikzpicture}
    \caption{Bump functions}
    \label{fig: bump function}
\end{figure}
Let $\psi_i \in L^2_1(u_i^* TM)$ be orthogonal to $\ker D_{u_i}$.
Then we construct the pre-glued curve
\begin{equation}\label{eqn: u}
u = \sum_i \beta_i(u_i + \psi_i).
\end{equation}
Then $$\mathcal L u = \sum_i \beta_i \Theta_i,$$
where 
    \begin{equation}\label{eqn: Thetai}
    \Theta_i = D_{u_i} \psi_i + \frac{d \beta^+_{i-1}}{d \vt} (u_{i-1}+ \psi_{i-1}) + \frac{d \beta^-_{i+1}}{d\vt} (u_{i+1} + \psi_{i+1}).
    \end{equation}
The bump functions $\beta_i^+$ satisfy:
\be 
    \item $\beta_i^+(\vt) = \beta_i(\vt)$ for $\vt \geq a_i-3T_i + r$;
    \item $\beta_i^+(\vt) = 1$, elsewhere,
\ee  
and $\beta_i^-$ satisfies:
\be 
    \item $\beta_i^-(\vt) = \beta_i(\vt)$ for $\vt \leq b_i + 3T_{i+1} - r$;
    \item $\beta_i^-(\vt) = 1$, elsewhere.
\ee  
In particular, $\beta_i = \beta_i^- \beta_i^+$.
    
Using the convention that $a_1 = -\infty$ and $b_m = \infty$, we have
$$\Theta_1 = D_{u_1} \psi_1 +  \frac{d \beta_{2}^-}{d\vt} (u_{2} + \psi_{2})$$
and 
$$\Theta_m = D_{u_m} \psi_m + \frac{d \beta^+_{m-1}}{d \vt} (u_{m-1}+ \psi_{m-1}).$$

Let $\Pi_i: L_1^2(u_i^* TM) \to \ker D_{u_i}^*$ be the orthogonal projection for all $i=1,\ldots, m$.
\begin{lemma}
    There exist sufficiently large $r$ and $R$, such that for any $T_1, \dots, T_{m-1} > R$, the following system of equations:
    \begin{equation}\label{eqn: for psi}
    D_{u_i} \psi_i +  (1 - \Pi_i ) \left(\frac{d \beta^+_{i-1}}{d \vt} (u_{i-1}+ \psi_{i-1}) + \frac{d \beta^-_{i+1}}{d\vt} (u_{i+1} + \psi_{i+1})\right) = 0,
    \end{equation}
    for $i = 1, \dots, m$,
    has a unique solution $(\psi_1, \dots, \psi_m) \in \ker D_{u_1}^\perp \times \dots \times \ker D_{u_m}^\perp$.
\end{lemma}
\begin{proof}
    This is essentially the same as Lemma 5.6 in \cite{hutchings2009gluing}. One can show that by fixing $||(\psi_1, \dots, \psi_{m})||_{L^2_1} \leq C$ for some small constant $C$, the map 
    \begin{align*}
    (\psi_1, \dots, \psi_{m}) \mapsto &  (-D_{u_1}^{-1} (1 - \Pi_1 )  \frac{d \beta^-_{2}}{d\vt} (u_{2} + \psi_{2}),\\
    & \dots \\
    &  -D_{u_m}^{-1} (1 - \Pi_m )  \frac{d \beta^+_{m-1}}{d \vt} (u_{m-1}+ \psi_{m-1}))
    \end{align*}
    is a contraction map, and hence has a unique solution,
    where $D_{u_i}^{-1}$ is the unique pre-image that is perpendicular to $\ker D_{u_i}$.
\end{proof}
Denote $\lm_i = \lm(\cP{p}_{i-1}, \cP{p}_i)$.
We now define the obstruction section $\obstructionSectionS = (\obstructionSectionS^1, \dots, \obstructionSectionS^m)$ of the bundle $$\arr{\mathcal O} \to \arr{\mathcal M}_R,$$ 
by
\begin{equation}\label{eqn: def of s}
    \obstructionSectionS^i = \Pi_i  \left(\frac{d \beta^-_{i-1}}{d \vt} (u_{i-1}+ \psi_{i-1}) + \frac{d \beta^+_{i+1}}{d\vt} (u_{i+1} + \psi_{i+1})\right),
\end{equation}
where 
\be 
    \item $\arr{\mathcal M}_R \coloneqq \lm_1 \times \dots \times \lm_m \times [R, \infty)^{m-1}$,
    \item $\arr{\mathcal O} = \oplus_{i=1}^m \op{pr}_i^* \mathcal O_i$, and 
    \item $\op{pr}_i: \arr{\mathcal M}_R \to \lm_i$ is the projection map.
\ee

Then we can construct the gluing map $G: \obstructionSectionS^{-1}(0) \to \lm(\cP{p}, \cP{q})$ sending $$([u_1], \dots, [u_m], T_1, \dots , T_{m-1}) \mapsto u.$$

The section $\obstructionSectionS$ is defined implicitly as it involves solving for $(\psi_1, \dots, \psi_m)$. 
Instead, we define a linearized obstruction section $\obstructionSectionS_0 = (\obstructionSectionS_0^1,\dots, \obstructionSectionS_0^m)$, where 
$$\obstructionSectionS_0^i = \Pi_i  \left(\frac{d \beta^-_{i-1}}{d \vt} u_{i-1} + \frac{d \beta^+_{i+1}}{d\vt} u_{i+1}\right).$$
As we will show in Theorem~\ref{thm: gluing}, the zeroes of $\obstructionSectionS_0$ are cobordant to the zeroes of $\obstructionSectionS$. 
If the $\dim \ker D_{u_i}^* = 1$, we can find a nonzero element $0 \neq \eta_i \in \ker D_{u_i}^*$ such that $||\eta_i||_{L^2} = 1$. Then 
\begin{align*}
\obstructionSectionS_0^i & =  \left \langle \frac{d \beta^-_{i-1}}{d \vt} u_{i-1} + \frac{d \beta^+_{i+1}}{d\vt} u_{i+1}, \eta_i  \right \rangle \eta_i \\
& =  \left(\int_{\R} \frac{d \beta^-_{i-1}}{d \vt}\langle  u_{i-1},  \eta_i \rangle_g + \frac{d \beta^+_{i+1}}{d\vt} \langle u_{i+1}, \eta_i \rangle_g d\vt \right) \eta_i 
\end{align*}
where $\langle \cdot, \cdot \rangle_g$ is the inner product with respect to $g$.
Now we study the zeroes of $\obstructionSectionS_0^i$. 
Note that $\frac{d \beta^-_{i-1}}{d \vt}\langle  u_{i-1},  \eta_i \rangle_g$ is supported near the critical point $\cP{p}_{i-1}$, and since $\frac{d \beta^-_{i-1}}{d \vt} \leq 0$, it is of the opposite sign, inside its support, to $\langle u_{i-1}^\infty, \eta_i^{-\infty}\rangle$.
Similarly, since $\frac{d \beta^+_{i+1}}{d\vt} \geq 0$, $\frac{d \beta^+_{i+1}}{d\vt} \langle u_{i+1}, \eta_i \rangle_g$ is of the same sign inside its support as $\langle \eta_i^{\infty}, u_{i+1}^{-\infty} \rangle$.

Therefore, if $\langle u_{i-1}^\infty, \eta_i^{-\infty}\rangle \cdot \langle \eta_i^{\infty}, u_{i+1}^{-\infty} \rangle < 0$ and if $R$ is sufficiently large, there does not exist $(T_1, \dots, T_{m-1})\in [R, \infty)^{m-1}$ such that $$([u_1], \dots, [u_m], T_1, \dots , T_{m-1}) \in \obstructionSectionS_0^{-1}(0).$$

On the other hand, since $\cP{p}_{i-1}$ is a non-degenerate critical point, and $\mathcal N_{\cP{p}_{i-1}}$ is an eigenspace of the Hessian of $f$ at $\cP{p}_{i-1}$, the normal component of $-\op{grad}_{f,g}$ is nonzero in $U \backslash M^G$, where $U$ is a small neighborhood of $\cP{p}_{i-1}$. This implies that the normal component $n_{i-1}$ of $u_{i-1}$ converges to $0$ monotonically as $s \to \infty$. 
Note that $\eta_i$ as a solution of the linear ordinary equation also converges to $0$ monotonically as $\vt \to -\infty$. This implies $\int_{\R} \frac{d \beta^-_{i-1}}{d \vt}\langle  u_{i-1},  \eta_i \rangle_g  d\vt$ converges to $0$ monotonically as $T_{i-1} \to \infty$. The same is true for the term $\int_{\R} \frac{d \beta^+_{i+1}}{d\vt} \langle u_{i+1}, \eta_i \rangle_g d\vt$. This means that if $\langle u_{i-1}^\infty, \eta_i^{-\infty}\rangle \cdot \langle \eta_i^{\infty}, u_{i+1}^{-\infty} \rangle > 0$, then for any sufficiently large $T_{i-1}$, there exists a unique $T_i$ such that $([u_1], \dots, [u_m], T_1, \dots , T_{m-1}) \in \obstructionSectionS_0^{-1}(0)$. We have just proved the following lemma:

\begin{lemma}\label{lemma: zero of linearized section}
    Given a gluable broken trajectory $([u_1], \dots, [u_m])$, an index $i \in \{1, \dots, m-1\}$ and any sufficiently large $T_i$, there \cbu exists an unique value \cb of $$(T_1, \dots, \widehat T_i, \dots, T_{m-1}),$$ such that $([u_1], \dots, [u_m], T_1, \dots , T_{m-1}) \in \obstructionSectionS_0^{-1}(0)$.
\end{lemma}

\begin{definition}[Close to breaking]\label{defn:close_to_breaking}
    A trajectory $[u] \in \lm(\cP{p}_0, \cP{p}_m)$ is \emph{$\delta$-close to breaking into the broken trajectory} $([u_1], [u_2], \dots, [u_m]) \in \lm_1 \times \lm_2 \times \dots \times \lm_m$ if there exist representatives $u, u_1, u_2, \dots, u_m$ of $[u], [u_1], [u_2], \dots, [u_m]$ respectively, and constants $a_1 < b_1 < \dots < a_{m} < b_{m}$ such that for each $k \in \{1, 2, \dots, m\}$, the following hold:
    \begin{enumerate}[label=(\alph*)]
        \item $u|_{[a_k, b_k]}$ and $u_k|_{[a_k, b_k]}$ are $\delta$-close in the $C^1$-norm.
        \item $u|_{(b_k, a_{k+1})}$ is $\delta$-close in the $C^1$-norm to the constant map $\cP{p}_k$, where $b_0 \coloneqq -\infty$ and $a_m \coloneqq \infty$.
    \end{enumerate}
    Denote the space of all $\delta$-close to breaking trajectories by $\lm^\delta_{\arr{\cP{p}}}(\cP{p_0}, \cP{p_m})$.
\end{definition}   

\begin{theorem}[Gluing]\label{thm: gluing}
    Assume the setting of Theorem~\ref{thm: gluing end goal}.
    Let $([u_1], [u_2], \dots, [u_m])$ be a gluable broken trajectory from $\cP{p}$ to $\cP{q}$.
    There exist a sufficiently large $r$, a gluing parameter bound $R > 0$, and a close-to-breaking parameter bound $\delta > 0$ such that the following holds:
    \begin{enumerate}[label=(\roman*)]
        \item The obstruction section $\obstructionSectionS$ of the bundle 
        $$\arr{\mathcal O} \to \arr{\mathcal M}_R$$ is $C^1$ and intersects the zero section transversely.
        \item \cbu There exists a gluing map $G: \obstructionSectionS^{-1}(0) \to \lm^\delta_{\arr{\cP{p}}}(\cP{p_0}, \cP{p_m})$ such that the following holds: \cb
        \begin{enumerate}[label=(\alph*)]
            \item $G$ is a \cbu homeomorphism \cb onto its image.
            \item For any $R' \geq R$, there exists $\delta' > 0$ such that $$\lm^{\delta'}_{\arr{\cP{p}}}(\cP{p_0}, \cP{p_m}) \subset G(\obstructionSectionS^{-1}(0) \cap \arr{\mathcal M}_{R'}).$$
            \item For any $0 < \delta' < \delta$, there exists $R'$ such that $$\lm^{\delta'}_{\arr{\cP{p}}}(\cP{p_0}, \cP{p_m}) \supset G(\obstructionSectionS^{-1}(0) \cap \arr{\mathcal M}_{R'}).$$
        \end{enumerate}
        \item \cbu For any $i = 1, 2, \dots, m$, and a generic $R' > R$, 
         $\obstructionSectionS^{-1}(0) \cap \{T_i = R'\}$ is cobordant to $\obstructionSectionS_0^{-1}(0) \cap \{T_i = R'\}$. \cb
    \end{enumerate}
\end{theorem} \cbu 
\begin{proof}[Sketch of proof]
    In the case of $J$-holomorphic curves, this is proven by Hutchings-Taubs in \cite{hutchings2009gluing}:
    Item (ii) follows from Theorem 7.3; Item (iii) is covered in Corollary 8.6; Item (i) is explained in the third paragraph of Section 10.2.

    We outline an alternative approach of (i) and (ii), following the proof of Theorem 3.5 of \cite{bao2024computable}, fitting this into the framework of Kuranishi gluing.
    Over a neighborhood $U$ of $\lm^\delta_{\arr{\cP{p}}}(\cP{p_0}, \cP{p_m})$ inside $\widetilde{\mathcal B}(\cP{p}_0, \cP{p}_m)$, one can define a sub-bundle $\widetilde{\mathcal O} \subset \widetilde{\mathcal E}$, so that for any $u = \sum_i \beta_i u_i$, we have $$\widetilde{\mathcal O}|_u = \Bigl\{\sum_i \beta_i \xi_i ~\Big|~ \xi_i \in \mathcal O_i|_{u_i}\Bigr\}.$$
    
    Since $\morsedbar$ is transverse to the zero section by assumption, it is transverse to $\widetilde{\mathcal O}$.
    This allows us to define the thickened moduli space $\mathbb V = \morsedbar^{-1}(\widetilde{\mathcal O})$. 
    Now we can construct the standard Kuranishi gluing bundle map $(G', G'_\sharp)$:
    \begin{equation}
    \begin{tikzcd}
        \arr{\mathcal O} \arrow[d] \arrow[r, "G'_\sharp"] & \widetilde{\mathcal O} \arrow[d] \\
        \arr{\mathcal M}_{R} \arrow[r, "G'"] & \mathbb V \arrow[u, bend left, "\morsedbar"].\\
    \end{tikzcd}
    \end{equation}
    Here, $G'$ is a homeomorphism onto its image, and $G'_\sharp$ is a bundle isomorphism over the image of $G'$.
    The Kuranishi gluing map $G'$ is defined similarly to the obstruction bundle gluing map $G$, except that instead of trying to solve $\Theta_i = 0$, it solves $\Theta_i \in \mathcal O_i',$ where $\Theta_i$ is defined in Equation~\eqref{eqn: Thetai}; $\mathcal O_i'$ is a bundle over $\mathcal M_i$ that is close to $\mathcal O_i$. The difference (due to $\psi_i$) goes to zero as $R \to \infty$. Hence, $\mathcal O'_i|_{u_i} \oplus \op{im}D_{u_i} = \widetilde{\mathcal E}_{u_i}$, and $\Theta_i \in \mathcal O_i'$ can always be solved.
    Injectivity and surjectivity of $G'$ are proved similarly in the $J$-holomorphic curve case in Section 7.6 and Section 7.7 of \cite{bao2015semi}. 
    
    Now we have the obstruction section $\obstructionSectionS = (G'_\sharp)^{-1} \circ \morsedbar \circ G'$. In other words, the obstruction section is simply the operator $\morsedbar$ written in the boundary chart. Since $\morsedbar$ is transverse to the zero section by assumption, $\obstructionSectionS$ is transverse to the zero section, which proves (i). Items (ii) and (iii) follow from the injectivity and surjectivity of $G'$. 
\end{proof} \cb

\begin{proof}[Proof of Theorem~\ref{thm: gluing end goal}]
    Theorem~\ref{thm: gluing end goal} follows from Theorem~\ref{thm: gluing} and Lemma~\ref{lemma: zero of linearized section}, modulo the sign correction. We now explain the sign correction. We restrict to the case $m=3$ to simplify the notation. The general case is almost the same proof with more complicated notation.
    First, recall that the orientation of the moduli space $\lm(\cP{p}_0, \cP{p}_3)$ is defined as follows: for any $u \in \lm(\cP{p}_0, \cP{p}_3)$ that is close to breaking, we orient $T_u \lm(\cP{p}_0, \cP{p}_3)$ so that the isomorphism
    \begin{equation}\label{eqn: orientation interior}
    T_{\cP{p}_0} \mathscr D_{\cP{p}_0} \simeq \R\langle \frac{du}{d\vt} \rangle \oplus T_u \lm(\cP{p}_0, \cP{p}_3) \oplus T_{\cP{p}_3} \mathscr D_{\cP{p}_3}
    \end{equation}
    is orientation-preserving. Here, $\frac{du}{d\vt}$ is the vector arising from translating the domain in the positive direction; and the isomorphism is given by the flow of the negative gradient vector field, where we use the fact that $\lm(\cP{p}_0, \cP{p}_3)$ is transversely cut out.
    Note that these tangent spaces are implicitly mapped to an unspecified point on $u$ by the flow. To simplify the notation, we omit the point and the identification.
    Let $n \in T_u \lm(\cP{p}_0, \cP{p}_3)$ be the outward-pointing normal vector.
    We define $\boldsymbol{o}_\p (u) \in \{1, -1\}$ such that the orientation of $T_u \lm(\cP{p}_0, \cP{p}_3)$ is given by $\boldsymbol{o}_\p (u) n$. In this way, $\boldsymbol{o}_\p(u)$ is assigned the boundary orientation of $\lm(\cP{p}_0, \cP{p}_3)$.
    We have three equations similar to Equation~\ref{eqn: orientation interior}:
    \begin{equation}\label{eqn: orientation u1}
        T_{\cP{p}_0}\mathscr D_{\cP{p}_0} \simeq \orientation{u_1} \R\langle \frac{du_1}{d\vt} \rangle \oplus T_{\cP{p}_1} \mathscr D_{\cP{p}_1},
    \end{equation}
    \begin{equation}\label{eqn: orientation u2}
        T_{\cP{p}_1}\mathscr D_{\cP{p}_1} \simeq \orientation{u_2} \R\langle \frac{du_2}{d\vt} \rangle \oplus T_{\cP{p}_2} (\mathscr D_{\cP{p}_2} \cap F_2),
    \end{equation}
    \begin{equation}\label{eqn: orientation u3}
        T_{\cP{p}_2}\mathscr D_{\cP{p}_2} \simeq \orientation{u_3} \R\langle \frac{du_3}{d\vt} \rangle \oplus T_{\cP{p}_3} \mathscr D_{\cP{p}_3}.
    \end{equation}
From these we obtain 
    \[
    \R\langle \mathfrak o_{\cP{p}_2} \rangle \oplus T_{\cP{p}_0} \mathscr D_{\cP{p}_0} \simeq \orientation{u_1} \orientation{u_2} \orientation{u_3} \R\langle \frac{du_1}{d\vt}, \frac{du_2}{d\vt}, \frac{du_3}{d\vt} \rangle,
    \]
    where we used $\R \langle \frac{du_2}{d\vt} \rangle \oplus T_{\cP{p}_2} (\mathscr D_{\cP{p}_2} \cap F_2) = T_{\cP{p}_2} \mathscr D_{\cP{p}_2}$.
    Comparing this with Equation~\eqref{eqn: orientation interior}, we obtain the isomorphism:
    \begin{equation}\label{eqn: orientation isomorphism}
    \boldsymbol{o}_\p(u) \R \langle \frac{du}{d\vt}, n \rangle \oplus \mathfrak o_{\cP{p}_2} \simeq \orientation{u_1} \orientation{u_2} \orientation{u_3} \R\langle \frac{du_1}{d\vt}, \frac{du_2}{d\vt}, \frac{du_3}{d\vt} \rangle,
    \end{equation}
    where the isomorphism is orientation-preserving, the vectors in $\langle \ \rangle$ are ordered. The left-hand side of Equation~\eqref{eqn: orientation isomorphism} determines an orientation of $\ker D_u \oplus \ker D_u^*$, and the right-hand side determines an orientation of $\ker D_u \oplus \ker D_u^*$ by gluing, as explained in Section 9 of \cite{hutchings2009gluing}. 
    Via gluing, modulo multiplication by a matrix of positive determinant, it is clear that $\frac{du_1}{d\vt} + \frac{du_2}{d\vt} + \frac{du_3}{d\vt} \mapsto \frac{du}{d\vt}$, and $-\frac{du_1}{d\vt} + \frac{du_3}{d\vt} \mapsto n$. Translating in the direction of $\frac{du_2}{d\vt}$ decreases the contribution of  $\frac{d \beta_1^+}{d \vt}(u_1 + \psi_1)$ and increases the contribution of $\frac{d \beta_3^-}{d \vt}(u_3 + \psi_3)$ in Equation~\eqref{eqn: def of s}, and hence $\frac{du_2}{d\vt} \mapsto \langle u_3^{-\infty}, \mathfrak o_{\cP{p}_2} \rangle \mathfrak o_{\cP{p}_2}$.
    But 
    \[
    (\frac{du_1}{d\vt} + \frac{du_2}{d\vt} + \frac{du_3}{d\vt}) \wedge (-\frac{du_1}{d\vt} + \frac{du_3}{d\vt}) \wedge \frac{du_2}{d\vt} = - \frac{du_1}{d\vt} \wedge \frac{du_2}{d\vt} \wedge \frac{du_3}{d\vt}.
    \]
    Therefore, $\boldsymbol{o}_\p(u) = - \orientation{u_1} \orientation{u_2} \orientation{u_3} \langle u_3^{-\infty}, \mathfrak o_{\cP{p}_2} \rangle$.

\end{proof}
\cbu 
\section{A Further Example}\label{section: beyond reflection}
In this section, we present another example to illustrate obstruction bundle gluing, which lies slightly beyond the reflection group setting. 
This discussion is \emph{not} rigorous and is largely conjectural. A rigorous presentation would require gluing broken trajectories of $f$ to smooth trajectories of $f^\epsilon$. 

We consider the upright, vertically embedded genus-two Riemann surface equipped with a Morse function given by the height function, as shown in Figure~\ref{fig: genus two surface}. 
The moduli spaces $\m{a}{b}$, $\m{b}{c}$, $\m{c}{d}$, and $\m{d}{e}$ are cleanly cut out, but not transversely cut out.

To address this, we modify the Morse function infinitesimally at a few critical points.
It is convenient to introduce a group action $G_1 \times G_2$,
where $G_1 = \Z_2$ acts by front-back reflection, and $G_2 = \Z_2$ acts by left-right reflection. 
Let $g$ denote the standard Euclidean metric restricted to the surface.
The pair $(f,g)$ is not stably Morse–Smale with respect to the $G_1$-action nor the $G_2$-action.
Although the pair $(f,g)$ is stably Morse–Smale with respect to the $G_1 \times G_2$-action,
proceeding in this direction would require significantly more work. 
Instead, we add a few generators to the chain complex. We call the collection of new generators and original critical points generalized critical points. Specifically, 
\begin{itemize}
    \item add $\cP{c}_+$ and $\cP{c}_-$, and set the gradings $|\cP{c}_\pm| = 1$ and $|\cP{c}| = 0$;
    \item add $\cP{d}_l$ and $\cP{d}_r$, and set the gradings $|\cP{d}_l| = |\cP{d}_r| = 1$ and $|\cP{d}| = 0$;
    \item add $\cP{e}_+$ and $\cP{e}_-$, and set the gradings $|\cP{e}_\pm| = 1$ and $|\cP{e}| = 0$.
\end{itemize}
Here, the subscripts $\pm$ indicate generators added with respect to the front-back $G_1$-action, and $l, r$ indicate those added with respect to the front-back $G_2$-action.
These generalized critical points can be interpreted as critical points of a perturbed Morse function $f^\epsilon$.
The chain complex $C$ is freely generated over $\Z$ by
\[
\{\cP{a}, \cP{b}, \cP{c}_+, \cP{c}_-, \cP{c}, \cP{d}_l, \cP{d}_r, \cP{d}, \cP{e}_+, \cP{e}_-, \cP{e}, \cP{f}\}.
\]

We now consider moduli spaces of generalized gluable broken (ggb) trajectories $\widehat{\mathcal M}(\cP{p}, \cP{p'})$ between any two generalized critical points.
A ggb trajectory is a broken trajectory that can be glued to a genuine trajectory of $f^\epsilon$. 
We define the differential $\p$ by counting index-$1$ ggb trajectories. 

To show $\p^2 = 0$, we examine the boundaries of moduli spaces of index-$2$ ggb trajectories. 
In this example, the relevant moduli spaces are 
\[
\widehat{\mathcal M}(\cP{a}, \cP{c}),\quad \widehat{\mathcal M}(\cP{a}, \cP{d}),\quad \widehat{\mathcal M}(\cP{a}, \cP{e}),\quad \text{and} \quad \widehat{\mathcal M}(\cP{a}, \cP{f}).
\]
The first three are empty. For instance, a candidate element in $\widehat{\mathcal M}(\cP{a}, \cP{c})$ such as the generalized broken trajectory $(u,v)$ is not gluable. On the other hand, $\widehat{\mathcal M}(\cP{a}, \cP{f})$ is non-empty, and its boundary is given by
\[
\coprod_{\cP{p}} \widehat{\mathcal M}(\cP{a}, \cP{p})\times  \widehat{\mathcal M}(\cP{p}, \cP{f}),
\]
where the disjoint union ranges over \emph{all} generalized critical points $\cP{p}$ such that $|\cP{a}| - |\cP{p}| = 1$. This accounts for why $\p^2 = 0$. 
In this case, $\widehat{\mathcal M}(\cP{p}, \cP{f}) = \emptyset$ for $\cP{p} \in \{\cP{b}, \cP{c}_+, \cP{c}_-, \cP{d}_l, \cP{d}_r\}$. 
Hence, the boundary of $\widehat{\mathcal M}(\cP{a}, \cP{f})$ is 
\[
\left(\widehat{\mathcal M}(\cP{a}, \cP{e}_+) \times \widehat{\mathcal M}(\cP{e}_+, \cP{f}) \right) \coprod \left(\widehat{\mathcal M}(\cP{a}, \cP{e}_-) \times \widehat{\mathcal M}(\cP{e}_-, \cP{f}) \right),
\]
which is 
\[
\{(u,v,w,x,y), (u,v',w,x',y), (u',v,w',x,y'), (u',v',w',x',y')\}.
\]

\begin{figure}[h]
    \centering 
    \begin{tikzpicture}[rotate = 90, scale = 0.9]
    \draw[blue, smooth] (0,1) to[out=30,in=150] (2,1) to[out=-30,in=210] (3,1) to[out=30,in=150] (5,1) to[out=-30,in=30] (5,-1) to[out=210,in=-30] (3,-1) to[out=150,in=30] (2,-1) to[out=210,in=-30] (0,-1) to[out=150,in=-150] (0,1);
    \draw[blue, smooth] (0.4,0.0) .. controls (0.8,-0.25) and (1.2,-0.25) .. (1.6,0.05);
    \draw[blue, smooth] (0.5,0) .. controls (0.8,0.2) and (1.2,0.2) .. (1.5,0);
    \draw[blue, smooth] (3.4,0.0) .. controls (3.8,-0.25) and (4.2,-0.25) .. (4.6,0.05);
    \draw[blue, smooth] (3.5,0) .. controls (3.8,0.2) and (4.2,0.2) .. (4.5,0);
    \node[blue] at (5.5, -1.5) {\footnotesize{$M^{G_1}$}};

    \draw[Green] (5.5, 0.0) .. controls (5.25,0.3) and (4.75, 0.3).. (4.5, 0);
    \node[] at (5, 0.5) {\footnotesize $u$};

    \draw[dotted, Green] (5.5, 0.0) .. controls (5.25,-0.3) and (4.75, -0.3).. (4.5, 0);
    \node[] at (5, -0.5) {\footnotesize $u'$};

    \draw[Green] (3.45, 0) .. controls (2.75,0.4) and (2.25, 0.4).. (1.5, 0);
    \draw[dotted, Green] (3.45, 0) .. controls (2.75,-0.4) and (2.25, -0.4).. (1.5, 0);

    \draw[Green] (.5, 0)  .. controls (.25,0.3) and (-0.25, 0.3).. (-0.5, 0) ;
    \draw[dotted, Green] (.5, 0)  .. controls (.25,-0.3) and (-0.25, -0.3).. (-0.5, 0) ;

    \node[Green] at (5.5, 1.5) {\footnotesize{$M^{G_2}$}};

    \draw[fill] (5.5, 0) circle (0.05);
    \node[anchor = south] at (5.5, 0) {$\cP{a}$};
    \draw[fill] (4.5, 0) circle (0.05);
    \node[anchor = south] at (4.5, 0) {$\cP{b}$};
    \draw[fill] (3.5, 0) circle (0.05);
    \node[anchor = north] at (3.5, 0)  {$\cP{c}$};
    \draw[fill] (1.5, 0) circle (0.05);
    \node[anchor = south] at (1.5, 0) {$\cP{d}$};
    \draw[fill] (.5, 0) circle (0.05);
    \node[anchor = north] at (.5, 0)  {$\cP{e}$};
    \draw[fill] (-0.5, 0) circle (0.05);
    \node[anchor = north] at (-0.5, 0)  {$\cP{f}$};

    \node[] at (4, 0.35) {\footnotesize $v$};
    \node[] at (4.05, -0.45) {\footnotesize $v'$};

    \node[] at (2.5, 0.5) {\footnotesize $w$};
    \node[] at (2.5, -0.55) {\footnotesize $w'$};

    \node[] at (1, 0.45) {\footnotesize $x$};
    \node[] at (1.05, -0.45) {\footnotesize $x'$};

    \node[] at (0, 0.5) {\footnotesize $y$};
    \node[] at (0, -0.5) {\footnotesize $y'$};
    \end{tikzpicture}
    \caption{The upright genus-two surface}
    \label{fig: genus two surface}
\end{figure}

We now explain how the notion of ggb trajectories applies in this example and show that $\p \cP{a} = 0$.
To compute $\langle \p \cP{a}, \cP{e}_+ \rangle$, the coefficient of $\cP{e}_+$ in $\p \cP{a}$, we count the ggb trajectories $(u, v, w, x)$ and $(u, v', w, x')$, which cancel each other. 
The trajectory $(u, v, w, x)$ is not a ggb trajectory from $\cP{a}$ to $\cP{e}$ because $x$ is obstructed. 
However, since
\[
\langle w^{\infty}, \xi^{-\infty} \rangle \cdot \langle \xi^\infty, \mathfrak o_{\cP{e}} \rangle > 0,
\]
$(u, v, w, x)$ is a ggb trajectory from $\cP{a}$ to $\cP{e}_+$. Similarly, 
\begin{itemize}
    \item[-] $\langle \p \cP{a}, \cP{e}_- \rangle$ counts $(u', v, w', x)$ and $(u', v', w', x')$; 
    \item[-] $\langle \p \cP{a}, \cP{d}_l \rangle$ counts $(u, v, w)$ and $(u', v, w')$; 
    \item[-] $\langle \p \cP{a}, \cP{d}_r \rangle$ counts $(u, v', w)$ and $(u', v', w')$; 
    \item[-] $\langle \p \cP{a}, \cP{c}_+ \rangle$ counts $(u, v)$ and $(u, v')$; 
    \item[-] $\langle \p \cP{a}, \cP{c}_- \rangle$ counts $(u', v)$ and $(u', v')$; 
    \item[-] $\langle \p \cP{a}, \cP{b} \rangle$ counts $u$ and $u'$.
\end{itemize}
For other terms, we have:
\begin{itemize}
    \item[-] $\p \cP{b} = 0$, since $\langle \p \cP{b}, \cP{c} \rangle = 0$ counts $v$ and $v'$, and there are no ggb trajectories to other generators;
    \item[-] $\p \cP{c}_+ = - \cP{c} + \cP{d}$, by counting $w$ and no ggb trajectories to other generalized critical points; similarly, $\p \cP{c}_- = \cP{c} - \cP{d}$;
    \item[-] $\p \cP{d}_l = - \cP{d} + \cP{e}$, from $x$ and the trivial trajectory from $\cP{d}_l$ to $\cP{d}$; similarly, $\p \cP{d}_r = \cP{d} - \cP{e}$;
    \item[-] $\p \cP{e}_+ = - \cP{e} + \cP{f}$, and $\p \cP{e}_- = \cP{e} - \cP{f}$.
\end{itemize}

The resulting homology of this genus-two Riemann  surface $\Sigma_2$ is \[H_0(\Sigma_2; \Z)=\Z, \quad \quad H_1(\Sigma_2; \Z)=\Z^4, \quad \quad H_2(\Sigma_2; \Z)=\Z.\]

\bibliographystyle{alpha}
\bibliography{mybib}

\end{document}